\documentclass[
final,%
11pt,%
a4paper,%
twoside,%
]{article}

\usepackage[top = 3cm, bottom = 3cm, left = 3cm, right =
3cm]{geometry}

\usepackage{fancyhdr}
\usepackage{etoolbox}

\usepackage[british]{babel}

\usepackage[utf8]{inputenc}
\usepackage[T1]{fontenc}

\usepackage[%
protrusion = true,
expansion = true,
auto = true,
kerning = true,
spacing = true,
]{microtype}

\usepackage{siunitx}

\usepackage{csquotes}

\usepackage[section]{placeins}
\makeatletter
\AtBeginDocument{%
  \expandafter\renewcommand\expandafter\subsection\expandafter{%
    \expandafter\@fb@secFB\subsection
  }%
}
\makeatother

\usepackage{caption}

\usepackage{mathtools,booktabs}
\usepackage[amsmath,thmmarks]{ntheorem}
\usepackage{amsfonts}
\usepackage{graphicx}
\usepackage{dsfont}
\usepackage{url}
\usepackage{bbm}
\usepackage{mathtools}

\numberwithin{equation}{section}

\allowdisplaybreaks

\usepackage{enumitem}

\newlist{extralist}{enumerate}{1}

\newlist{tlist}{enumerate}{1}

\setlist{
  font = \normalfont,%
  itemindent = 0pt,%
  leftmargin = *,%
  topsep = 0pt,%
}

\setlist[tlist]{%
  label = {(\roman{*})},%
}

\setlist[extralist]{
  label = (\arabic{*}),%
}

\usepackage{titling}
\usepackage[
hidelinks,%
bookmarksnumbered,%
]{hyperref}

\usepackage[%
backend = bibtex,%
maxcitenames = 2,%
giveninits = true,%
style = numeric,
]{biblatex}

\DeclareNameAlias{author}{family-given}

\renewbibmacro{in:}{}

\DeclareFieldFormat{pages}{#1}

\renewbibmacro*{volume+number+eid}{%
  \setunit{\addcomma\space}%
  \printfield{volume}%
  \printfield{number}%
  \setunit{\addcomma\space}%
  \printfield{eid}}

\DeclareFieldFormat[article]{number}{\textit{\mkbibparens{#1}}}

\DeclareFieldFormat[article]{title}{#1}
\DeclareFieldFormat[book]{title}{\textit{#1}}

\usepackage[nameinlink,capitalise,nosort]{cleveref}

\crefname{enumi}{}{}

\theoremstyle{plain}
\theoremseparator{.}
\newtheorem{proposition}{Proposition}[section]

\newtheorem{theorem}[proposition]{Theorem}
\newtheorem{lemma}[proposition]{Lemma}

\theoremheaderfont{\itshape}
\theorembodyfont{}
\newtheorem{example}[proposition]{Example}
\newtheorem{remark}[proposition]{Remark}

\theoremstyle{nonumberplain}
\theoremheaderfont{\itshape}
\theorembodyfont{\normalfont}
\theoremseparator{.}
\theoremsymbol{\nolinebreak[1]\hspace*{.5em plus 1fill}\ensuremath\square}
\newtheorem{proof}{Proof}

\theoremstyle{empty}
\theoremheaderfont{\bfseries}
\theorembodyfont{\normalfont}
\theoremsymbol{\nolinebreak[1]\hspace*{.5em plus 1fill}\ensuremath\square}
\theoremseparator{.}

\makeatletter
\newtheoremstyle{named}{}{%
\item[%
  \hskip\labelsep%
  \theorem@headerfont%
  ##3%
  \theorem@separator%
  ]%
}
\makeatother
\theoremstyle{named}
\theoremseparator{.}
\theoremsymbol{}
\newtheorem{named}{}

\crefname{equation}{}{}
\crefname{tlisti}{}{}
\crefname{extralisti}{}{}
\crefname{page}{page}{pages}

\newcommand{\cL}{\mathcal{L}}

\newcommand{\bbP}{\mathbb{P}}

\newcommand{\bbN}{\mathbb{N}}
\newcommand{\bbZ}{\mathbb{Z}}
\newcommand{\bbR}{\mathbb{R}}

\DeclareMathOperator*{\argmin}{argmin}

\newcommand{\1}{\mathds{1}}

\providecommand\given{}
\DeclarePairedDelimiterX\Set[1]\{\}{%
  \renewcommand\given{\SetSymbol[\delimsize]}
  #1
}

\newcommand{\argmrk}{\cdot}

\DeclarePairedDelimiterXPP\Expt[1]{\mathbb{E}}[]{}{
  \renewcommand\given{\nonscript\:\delimsize\vert\nonscript\:\mathopen{}}
  #1}

\DeclarePairedDelimiterXPP\abs[1]{\mathop{}}\lvert\rvert{}{#1}

\DeclarePairedDelimiterXPP{\norm}[1]{\mathop{}}{\lVert}{\rVert}{}{%
  \ifblank{#1}{\argmrk}{#1}%
}

\DeclarePairedDelimiter{\floor}{\lfloor}{\rfloor}

\DeclarePairedDelimiter{\iprod}{\langle}{\rangle}

\newcommand{\di}{\textup{d}}
\newcommand{\isp}{\,}
\newcommand{\Di}{\isp\di}

\DeclareMathOperator{\var}{Var}
\DeclareMathOperator{\cov}{Cov}

\newcommand{\I}{\mathrm{i}}
\newcommand{\e}{\mathrm{e}}

\newcommand{\asTo}[1][]{\xrightarrow[#1]{\text{\tiny a.s.}}}
\newcommand{\distTo}[1][]{\xrightarrow[#1]{\smash{\,\scriptscriptstyle\mathcal{L}\,}}}
\newcommand{\fidiTo}[1][]{\xrightarrow[#1]{\scriptscriptstyle\textup{fidi}}}

\newcommand{\qtq}[1]{\quad\text{#1}\quad}
\newcommand{\iffx}{\quad\Longleftrightarrow\quad}

\makeatletter
\newcommand\thefontsize{The current font size is: \f@size pt}
\makeatother

\begin{document}
\title{Multi-Dimensional Parameter Estimation of Heavy-Tailed Moving
  Averages}

\author{%
  Mathias Mørck Ljungdahl%
  \\
  \itshape Department of Mathematics, Aarhus University
  % \\
  % E-mail: ljungdahl@math.au.dk%
  \\[2em]  
  Mark Podolskij
  \\
  \itshape Department of Mathematics, University of Luxembourg
  \\
  % E-mail: mpodolskij@math.au.dk%
}
\date{}
\maketitle

\begin{abstract}
  \noindent In this paper we present a parametric estimation method
  for certain multi-parameter heavy-tailed Lévy-driven moving
  averages. The theory relies on recent multivariate central limit
  theorems obtained via Malliavin calculus on Poisson spaces. Our
  minimal contrast approach is related to previous papers, which
  propose to use the marginal empirical characteristic function to
  estimate the one-dimensional parameter of the kernel function and
  the stability index of the driving Lévy motion. We extend their
  work to allow for a multi-parametric framework that in particular
  includes the important examples of the linear fractional stable
  motion, the stable Ornstein--Uhlenbeck process, certain
  CARMA($2, 1$) models and Ornstein--Uhlenbeck processes with a
  periodic component among other models. We present both the
  consistency and the associated central limit theorem of the
  minimal contrast estimator. Furthermore, we demonstrate numerical
  analysis to uncover the finite sample performance of our method.

  \bigbreak
  
\end{abstract}

\section{Introduction}

Steadily through the last decades estimation procedures for various
classes of continuous time moving averages and related processes have
been proposed, see, e.g. \cite{AyacLine,GrahScal, MazuEsti} for
estimation of the parameters in the linear fractional stable motion
model and \cite{DahlEffi,DangEsti} for the more general class of
self-similar processes among many others. The bedrock of these
techniques are of course the underlying limit theory for various
functionals of the processes~at~hand. One such seminal paper is
\cite{PipiCent}, which gives conditions for bounded functionals of a
large class of moving averages and was later extended in
\cite{PipiBoun} to certain unbounded~functions. In~a similar framework
\cite{BassPowe} gives an almost complete picture of the \enquote{law
  of large numbers} for the classical case of the power variation
functionals. The article \cite{BassOnLi} extends the functionals from
power variation to a large class of statistically interesting
functionals and for a class of symmetric $\beta$-stable moving
averages. This paper also provides an almost complete picture of the
corresponding weak limit theorems, at least in the setting of Appell
rank $> 1$ (such~as is the case for power variation and the (real
part) of the characteristic function).

Previous estimation methods suggested in
\cite{LjunAMin,LjunANot,MazuEsti} relied on functionals of the
one-dimensional marginal law of the process and specific properties of
the process at hand. Since~the marginal distribution of the considered
models have been symmetric $\beta$-stable, only the scale and the
stability parameters can be estimated via such statistics. In
particular, they are typically not sufficient to estimate kernel
functions that depend on a multi-dimensional parameter, which excludes
many interesting~models. Indeed, this discrepancy is observed in
\cite{LjunAMin}, where the characteristic function of the
one-dimensional law is not sufficient and instead the authors have to
rely on a combination with other statistics to ensure estimation of
all parameters.

The aim of this paper is to construct estimators of the kernel
function and the stability index in the general setting of a
multi-dimensional parameter space. Instead of relying on existing
theory \cite{BassOnLi,BassPowe,PipiCent}, which only accounts for the
marginal law of the underlying model, we shall use the framework from
the recent paper~\cite{AzmoMult}, which is tailor-made for the study
of Gaussian fluctuations of functionals of multiple heavy-tailed
moving averages, to estimate the multi-dimensional parameter. While
our approach is similar in spirit to the univariate framework of
\cite{LjunAMin,LjunANot} from the theoretical viewpoint, there are
some important differences. First of all, the parameter
identifiability and non-degeneracy condition
(see~Condition~(A)\cref{it:ass:A:4} below) are not trivial in the
multi-dimensional setting and we demonstrate the corresponding
techniques for various examples.  Secondly, the theoretical results of
our paper shed light on parameter estimation for a large class of
models; in particular, we~will present statistical inference for a
certain CARMA($2, 1$) model, which is novel in the
literature. Finally, we~remark that while \cref{thm:MCE-conv} is a
multivariate extension of \cite[Proposition~1]{LjunANot}, its~proof is
somewhat more complex than in the univariate setting.

Let us now define the class of moving average processes for which the
underlying limit theory applies. Let $L = (L_t)_{t \in \bbR}$ be a
standard symmetric $\beta$-stable Lévy process and consider the
model
\begin{equation}
  \label{eq:MA}
  X_t = \int_{-\infty}^{t} g(t - s) \Di L_s, \qquad t \in \bbR,
\end{equation}
for some measurable $g : \bbR \to \bbR$. Necessary and sufficient
conditions for the integral to exist are given in \cite{RajpSpec} and
we mention that in our setting a sufficient condition is
$\int_{\bbR} \abs{g(s)}^\beta \Di s < \infty$. The kernel function $g$
is assumed to have a power behaviour around $0$ and at
infinity. More~specifically, we shall assume the existence of a
constant $K > 0$ together with powers $\alpha > 0$ and
$\kappa \in \bbR$ for which it holds
\begin{equation}
  \label{eq:kernel-ass}
  \abs{g(x)} \leq K \bigl(x^{\kappa} \1_{[0, 1)}(x) + x^{-\alpha} \1_{[1, \infty)}(x) \bigr)
  \quad \text{for all $x \in \bbR$.}
\end{equation}
We are interested in (scaled) partial sums of multivariate functionals
of the vectors
$((X_{s + 1},\allowbreak \ldots,\allowbreak X_{s + m}))_{s \geq 0}$:
\begin{equation}
  \label{eq:main-statistic}
  V_n(X; f) = \frac{1}{\sqrt{n}} \sum_{s = 0}^{n-m} \left(f(X_{s + 1}, \ldots, X_{s + m}) - \Expt{f(X_1, \ldots, X_m)} \right),
\end{equation}
where $f : \bbR^m \to \bbR^d$ is a suitable Borel function. Adhering
to \cite[Remark~2.4(iii)]{AzmoMult} the following result holds. Below
$C_b^2(\bbR^m, \bbR^d)$ denotes the space of twice differentiable
functions $f:\bbR^m \to \bbR^d$ such that $f$ and all of its first and
second order derivatives are bounded and continuous.

\begin{theorem}[{{\cite[Theorem~2.3]{AzmoMult}}}]%
  \label{thm:MA-CLT}%
  Let $(X_t)_{t \in \bbR}$ be a moving average as in \cref{eq:MA} with
  kernel function $g$ satisfying \cref{eq:kernel-ass}. Assume that
  $\alpha \beta > 2 $ and $\kappa > - 1/\beta$. Let
  $f = (f_1, \ldots, f_d) \in C_b^2(\bbR^m, \bbR^d)$ and consider the
  statistic $V_n(X; f)$ introduced at \cref{eq:main-statistic}. Then
  as $n \to \infty$
  \begin{equation}
    \label{eq:asymp-cov}
    \Sigma_n^{i,j} \coloneqq \cov(V_n(X; f)) \to \Sigma^{i, j}
    \coloneqq \sum_{s \in \bbZ} \cov(f_i(X_{s + 1}, \ldots, X_{s + m}), f_j(X_1, \ldots, X_m))
  \end{equation}
  for any $1\leq i,j \leq d$. Moreover, $V_n(X; f) \distTo \mathcal{N}_d(0, \Sigma)$ as $n \to \infty$.
\end{theorem}

\noindent The paper \cite{AzmoMult} additionally provides
Berry--Esseen type bounds for an appropriate distance between
probability laws on $\bbR^d$, but \cref{thm:MA-CLT} is sufficient for
our statistical analysis. We remark that the limit theory for bounded
$f$ in the case of $m = 1$ and general $d \in \bbN$ is handled in
\cite{PipiBoun}, but it is actually the reverse situation, i.e.
$m \in \bbN$ and $d = 1$, that is the most needed. Specifically, $f$
will be the empirical characteristic function of the joint
distribution $(X_{s + 1}, \ldots, X_{s + m})$, which then grants us
the ability to estimate parameters which are not determined by the
one-dimensional distribution of $X_1$, see
\crefrange{ex:ornstein-uhlenbeck}{ex:CARMA} below.

The paper is organized as follows. In \cref{sec2} we introduce the
parametric model, numerous assumptions and
the main theoretical results of the paper, which show the strong
consistency and the asymptotic normality of the minimal contrast
estimator. \Cref{sec3} is devoted to a numerical analysis of the
finite sample performance of our estimator. Finally, all proofs are
collected in \cref{sec4}.

\section{The setting and main results} \label{sec2}

\subsection{The model and assumptions}

In the following we will consider a Lévy-driven moving average
$X = (X_t)_{t \in \bbR}$ given by
\begin{equation}
  \label{eq:MA-parametric}
  X_t = \int_{\bbR} g_{\beta, \theta}(t - s) \Di L_s, \qquad t \in \bbR,
\end{equation}
where $L$ is a symmetric $\beta$-stable Lévy process with unit scale
and $\beta \in \Upsilon$ for some open subset
$\Upsilon \subseteq (0, 2)$, and
$\Set{g_{\beta, \theta} \given \beta \in \Upsilon, \theta \in \Theta}$
is a measurable family of functions parametrized by an open subset
$\Upsilon \times \Theta \subseteq (0, 2) \times \bbR^d$ for some
$d \geq 1$. For ease of notation we shall often denote the joint
parameter with $\xi = (\beta, \theta)$ and the open subset by
$\Xi = \Upsilon \times \Theta$.

The main goal of this section is to extend the theory of
\cite{LjunANot} from a one-dimensional parameter space, i.e. $d = 1$,
to a general multi-dimensional theory. Such multi-dimensional
parameter spaces include important examples of the linear fractional
stable motion, the~stable Ornstein--Uhlenbeck process, certain
CARMA($2, 1$) models, and Ornstein--Uhlenbeck processes with a
periodic component among others. One of the main difficulties in
extending from $d = 1$ to $d \in \bbN$ is that, quite naturally,
the~parameters $(\beta, \theta)$ should be identifiable from the
(theoretical) statistic, which in the case of \cite{LjunANot} is the
one-dimensional characteristic function:
\begin{equation*}
  \phi_{\beta, \theta}(u) = \Expt{\e^{\I u X_1}}
  = \exp(-\norm{u g_{\beta, \theta}}_{\beta}^{\beta}).
\end{equation*}
This identification can very well be an unreasonable assumption if
$d > 1$, see \cref{ex:ornstein-uhlenbeck}. But if we instead consider
the characteristic function of the joint distribution
$(X_1, \ldots, X_m)$,
\begin{equation}
  \label{eq:joint-char-maps}
  \varphi^m_{\beta, \theta}(u_1, \ldots, u_m) = \Expt[\big]{\e^{\I \sum_{k = 1}^{m} u_k X_k}}
  = \exp\Bigl(-\norm[\Big]{\sum_{k = 1}^{m} u_k g_{\beta, \theta}(\argmrk + k)}_{\beta}^{\beta}\Bigr),
\end{equation}
such an identification may be possible. Let us discuss this in more
details. The~underlying stability index $\beta$ is always identifiable
from \cref{eq:joint-char-maps}, since the stability index of a stable
random variable is unique. The problem is then reduced to whether the
parametrization of the kernel $\theta \mapsto g_{\beta, \theta}$
specifies the distribution of $X$ uniquely. The question now becomes a
matter of uniqueness for the spectral representation of moving
averages, which has been studied in, e.g. \cite{RosiOnUn}. Translating
the question to the characteristic functions of the finite dimensional
distributions, $(X_1, \ldots, X_m)$, $m \in \bbN$, we ask whether the
$\beta$-norm of linear combinations of translations of the kernel
specifies $g_{\beta, \theta}$ uniquely. This is known as
\emph{Kanter's theorem} in the literature and first appeared in
\cite{KantLpNo}, but for exposition sake let us repeat it
here. Suppose~$\beta \in (0, \infty)$ is not an even integer and let
$g, h \in \cL^\beta(\bbR)$. Then Kanter's~theorem states that if for
all $n \in \bbN$ and $u_1, t_1, \ldots, u_n, t_n \in \bbR$ it holds
that
\begin{equation*}
  \norm[\Big]{\sum_{i = 1}^{n} u_i g(\argmrk + t_i)}_\beta^\beta
  = \norm[\Big]{\sum_{i = 1}^{n} u_i h(\argmrk + t_i)}_\beta^\beta,
\end{equation*}
then there exists an $\epsilon \in \Set{\pm 1}$ and a $\tau \in \bbR$
such that $g = \epsilon h(\argmrk + \tau)$ almost everywhere.
Kanter's~theorem then implies that the distribution of $X$ is the same
under $\theta$ and $\theta'$ if and only if there exists
$\epsilon \in \Set{\pm 1}$ and $\tau \in \bbR$ such that
\begin{equation*}
  \epsilon g_{\beta, \theta}(\argmrk + \tau) = g_{\beta, \theta'} \qquad \text{almost everywhere.}
\end{equation*}
For many concrete examples of the kernel family
$\Set{g_\xi \given \xi \in \Xi}$ it is often straightforward to check
that such an identity only occurs if $\epsilon = 1$, $\tau = 0$ and
$\theta = \theta'$.

Due to the preceding discussion it is reasonable to make the following
assumptions on the family of kernels and we note that similar
identification requirements are often explicitly or implicitly
required in the literature. An important remark is that our theory
allows for a general $m \in \bbN$ instead of only $m \in \Set{1, 2}$,
where the statistics in the case $m = 2$ are often
autocorrelations. We denote by $\partial_{\smash{\xi}}^{i, j} f_\xi$
the partial derivative of $f$ with respect to the $i$th and the $j$th
parameters $\xi_i$ and $\xi_j$ evaluated at $\xi \in \Xi$. Note in
particular that $\partial_{\smash{\xi}}^1$ is the derivative with
respect to $\beta$ and $\partial_\xi^i$ for $i \geq 2$ is the
derivative with respect to the $\theta_i$-coordinate according to the
notational convention given below \cref{eq:MA-parametric}.
\begin{named}[Assumption (A)]
  \label{ass:A}
  There exists an $m \in \bbN$ such that:
  \begin{extralist}[label = (\arabic*)]
  \item\label{it:ass:A:1}
    $0 < \norm{g_{\beta, \theta}}_\beta < \infty$ for all
    $(\beta, \theta) \in \Upsilon \times \Theta$.

  \item\label{it:ass:A:2} The map
    $\theta \mapsto \smash{\varphi_{\beta, \theta}^m}$ given in
    \cref{eq:joint-char-maps} is injective.

  \item\label{it:ass:A:3} The function
    $(\beta, \theta) \mapsto \norm{\sum_{i = 1}^{m} u_i g_{\beta,
        \theta}(\argmrk + i)}_{\beta}^{\beta}$ is
    $C^2(\Upsilon \times \Theta)$ for each
    $u_1, \ldots, u_m \in \bbR$.

  \item\label{it:ass:A:4}
    $u \mapsto \partial_{\xi}^1 \varphi^m_{\xi}(u), \partial_{\xi}^2
    \varphi_{\xi}^m(u), \ldots, \partial_{\xi}^{d + 1} \varphi^m_{\xi}(u)$
    are linearly independent continuous functions.
  \end{extralist}
\end{named}

\noindent Let us give some remarks about the imposed conditions.

\begin{remark}
  \label{rem:ass:A}%
  \leavevmode%
  \begin{enumerate}[label = (\roman{*})]
  \item\label{it:rem:ass:A:1} The assumption (A)\cref{it:ass:A:1} is a
    necessary and sufficient condition for $X$ to be well-defined and
    non-degenerate. Moreover, (A)\cref{it:ass:A:1} makes it apparent
    why an explicit dependence on $\beta$ of the kernel
    $g_{\beta, \theta}$ could be useful. This case of dependence is
    also necessary for some processes such as increments of the linear
    fractional stable motion, see~\cref{ex:lfsm} below.

  \item\label{it:rem:ass:A:2} Condition~(A)\cref{it:ass:A:2} is
    necessary to ensure that the model \cref{eq:MA-parametric} is
    parametrized properly. Note that the non-existence of an
    $m \in \bbN$ such that (A)\cref{it:ass:A:2} holds would imply that
    the parameters could never be inferred from any \emph{finite} data
    sample making the inference of $\theta$ impossible in
    practice. The identification of the parameters in a continuous
    time model from samples at equidistant time points is known in the
    literature as the \emph{aliasing problem}.

  \item\label{it:rem:ass:A:3} Condition (A)\cref{it:ass:A:3} is a
    minimal requirement for our method of proof (see also
    \cite[Assumption~(A)]{LjunANot}). In particular, it ensures
    existence of the derivatives in (A)\cref{it:ass:A:4}.

  \item\label{it:rem:ass:A:4} We note here that under
    assumption~(A)\cref{it:ass:A:1} 
    condition (A)\cref{it:ass:A:4} follows from linear independence of the
    smaller subset:
    $u \mapsto \partial_{\xi}^2 \varphi_{\xi}^m(u), \ldots,
    \partial_{\xi}^{d + 1} \varphi^m_{\xi}(u)$. Indeed, this follows
    from the particular form of the $\beta$-derivative as shown in
    \cref{sec:remark-proof}.
  \end{enumerate}
\end{remark}

\noindent In order to use \cref{thm:MA-CLT} we need to make additional
assumptions on our kernel and for this we need to introduce some more
notation. Consider a strictly positive weight function
$w \in \cL^1(\bbR^m_+)$ and define the weighted inner product and
norms
\begin{equation*}
  \iprod{g, h}_w = \int_{\bbR_+^m} g(x) h(x) w(x) \Di x \qtq{and}
  \norm{h}_{w, p}^p = \int_{\bbR_+^m} \abs{h(x)}^p w(x) \Di x, \qquad p \in \Set{1, 2}.
\end{equation*}
Let $\cL^p_w(\bbR^m_+)$ denote the corresponding Banach $\cL^p$-space
of Borel functions.

\begin{named}[Assumption~(B)]
  \label{ass:B}
  \leavevmode
  \begin{extralist}
  \item\label{it:ass:B:1} Assume that for all
    $(\beta, \theta) \in \Upsilon \times \Theta$ there exist
    $\kappa \in \bbR$ and $\alpha > 0$ such that $\kappa > -1/\beta$,
    $\alpha \beta > 2$ and \cref{eq:kernel-ass} holds for
    $g_{\beta, \theta}$.

  \item\label{it:ass:B:2} The functions
    $u \mapsto \abs{\partial_{\smash{\xi}}^{i, k} \varphi_{\xi}(u)},
    \abs{\partial_{\smash{\xi}}^{i} \varphi_{\xi}(u)}$,
    $i, k \in \Set{1, \ldots, d + 1}$, are locally dominated in
    $\cL_w^2(\bbR_+^m)$. That is, there exists for all $\xi \in \Xi$ a
    neighbourhood $\Xi_0 \ni \xi$ such that the supremum of these
    functions over $\xi \in \Xi_0$ are dominated by a function in
    $\cL_w^2(\bbR_+^m)$.
  \end{extralist}
\end{named}

\noindent
Assumption~(B)\cref{it:ass:B:1} is imposed to ensure that we may
employ \cref{thm:MA-CLT}. While (B)\cref{it:ass:B:2} seems strict it
is always satisfied in the one-dimensional case $m = 1$ and we shall
need it to ensure validity of the implicit function theorem in our
setup.

We now demonstrate some examples, which satisfy Assumption~(A) for
$m \geq 2$ but not for $m = 1$.

\begin{example}[Stable Ornstein--Uhlenbeck process]
  \label{ex:ornstein-uhlenbeck}
  Let $(X_t)_{t \in \bbR}$ denote the $\beta$-stable
  Ornstein--Uhlenbeck process with parameter $\lambda > 0$ and scale
  parameter $\sigma > 0$. That is, $(X_t)_{t \in \bbR}$ is a
  stationary solution of the stochastic differential equation
  \begin{equation*}
    \Di X_t = -\lambda X_t \Di t + \sigma \Di L_t.
  \end{equation*}
  It has the representation \cref{eq:MA-parametric} with kernel
  function $g_\theta(u) = \sigma \exp(-\lambda u) \1_{(0, \infty)}(u)$
  and $\theta = (\sigma, \lambda) \in (0, \infty)^2$. It is clear that
  the one-dimensional characteristic function does not characterize
  the parameter $\theta$, hence Assumption~(A)\cref{it:ass:A:2} is not
  satisfied for $m = 1$. Consider therefore the case $m = 2$. Here the
  characteristic function is uniquely determined by $\theta$ if the
  $\beta$-norms are. Indeed, using the binomial series one may deduce
  the following formula:
  \begin{equation*}
    \norm{u_1 g_\theta + u_2 g_\theta(\argmrk + 1)}_\beta^\beta
    = \frac{\sigma^\beta}{\beta \lambda} \Bigl[u_2^\beta (1 - \exp(-\beta \lambda))
    + (u_1 + u_2 \exp(-\lambda))^\beta \Bigr], \qquad u_1 > u_2 \geq 0.
  \end{equation*}
  It is then straightforward to check that these equations in
  $u_1 > u_2 \geq 0$ determine $\theta \in (0,
  \infty)^2$~uniquely. Additionally, (A)\cref{it:ass:A:4} can be
  checked in a manner similar to \cref{ex:periodic-example} below and
  we refer to \cref{sec:example-proofs} for the derivation of these
  statements.
  
  There are a number of alternative estimation methods for a stable
  Ornstein--Uhlenbeck model. When the stability parameter $\beta$ is
  known, $\lambda$ can be estimated with convergence rate
  $(n/\log n)^{1/\beta}$ as it has been shown in \cite{ZhanLeas}. In
  the discrete-time setting of the AR(1) model with heavy-tailed
  i.i.d. noise, it is known that a Gaussian limit can be obtained, cf.
  \cite{LingSelf}, but this method again lacks joint estimation with
  the parameter $\beta$. In a similar framework the paper
  \cite{AndrMaxi} investigates the asymptotic behaviour of the maximum
  likelihood estimator. In particular, their results imply that the
  parameters $\sigma$ and $\beta$ can be estimated with a
  $\sqrt{n}$-precision, while the drift parameter $\lambda$ has a
  faster convergence rate of $n^{1/\beta}$.
\end{example}

\begin{example}[Linear fractional stable motion]
  \label{ex:lfsm}
  Let $(Y_t)_{t \in \bbR}$ be the linear fractional stable
  motion with self-similarity $H \in (0, 1)$, stability index
  $\beta \in (0, 2)$ and scale parameter $\sigma > 0$. That is,
  \begin{equation*}
    Y_t = \int_{\bbR} \sigma [(t - s)_+^{H - 1/\beta}  -(-s)_+^{H - 1/\beta}] \Di L_s.
  \end{equation*}
  Consider the low frequency $k$th order increment at rate $r$ 
  ($k, r \in \bbN$) defined as 
  \begin{equation*}
    \Delta_{i, k}^r Y = \sum_{j = 0}^{k} (-1)^j \binom{k}{j} Y_{i - rj}, \qquad i \geq r k.
  \end{equation*}
  For example, for $k = 2$ we have that
  \begin{equation*}
    \Delta_{i, 2}^r Y = Y_i - 2 Y_{i - r} + Y_{i - 2r}, \qquad i \geq 2 r.
  \end{equation*}
  Our example process $X$ will be the $k$th increments at rate
  $r = 1$, i.e. $X_i \coloneqq \Delta_{i, k}^1 Y$. The~corresponding
  kernel of the process $(X_i)$ for a general $k$ is given by
  \begin{equation*}
    g_{\beta, H, \sigma}(u) \coloneqq \sum_{j = 0}^{k} (-1)^j \binom{k}{j} (u - j)_+^{H - 1/\beta},
  \end{equation*}
  where $x_+ = x \vee 0$ is the positive part and $x_+^a \coloneqq 0$
  for all $x \leq 0$. We note the asymptotic behaviour
  \begin{equation*}
    \frac{g_{\beta, H, \sigma}(u)}{K u^{H - 1/\beta - k}} \longrightarrow 1 \qquad \text{as $u \to \infty$}
  \end{equation*}
  for some constant $K > 0$ depending on $\alpha$, $H$ and $k$. Hence
  the kernel $g_{\beta, H, \sigma}$ for $X$ always satisfies the
  assumption~\cref{eq:kernel-ass} with $\kappa \coloneqq H - 1/\beta$
  and $\alpha \coloneqq k + 1/\beta - H$, since $H > 0$ and
  $k \geq 1 > H$.
  
  In this case Assumption~(B) for $g_{\beta, H, \sigma}$ can simply be
  translated into an assumption on the parameter space
  $\Upsilon \times \Theta$, e.g.
  \begin{equation*}
    \Upsilon \times \Theta = \Set{(\beta, H, \sigma) \given  0 < H < k - 1/\beta, \tfrac{1}{C} < \sigma < C},
  \end{equation*}
  for some arbitrary but finite constant $C > 0$. It is well-known
  that $X$ has a version with continuous paths if and only if
  $H - 1/\beta > 0$, so if we want to do inference in the continuous
  case we have the two parameter inequalities:
  \begin{equation}
    \label{eq:parameter-ineq}
    0 < H - 1/\beta \qtq{and} H < k - 1/\beta.
  \end{equation}
  We note that these inequalities never hold for $k = 1$, but they are
  always satisfied for $k \geq 2$. Indeed, the first inequality
  implies that $H \in (1/2, 1)$ and $\beta \in (1, 2)$.

  Now, the stability index $\beta$ and the scale parameter
  $\sigma > 0$ are identifiable from the one-dimensional
  characteristic function, since these parameters are unique in a
  stable distribution. The $H$-self-similarity of the linear
  fractional stable motion $Y$ implies that
  \begin{equation*}
    \frac{\Expt{\abs{\Delta_{2k, k}^2 Y}^p}}{\Expt{\abs{\Delta_{k, k} Y}^p}} = 2^{p H} \qquad \text{for } p \in (-1, 0).
  \end{equation*}
  For $k = 2$ the term $\Delta_{4, 2}^2 Y$ is a linear combination of
  $X_2 = \Delta_{2, 2}^1 Y$, $X_3 = \Delta_{3, 2}^1 Y$ and
  $X_4 = \Delta_{4, 2}^1 Y$. Hence $H$ is identifiable from the
  characteristic function of the three-dimensional distribution
  $(X_2, X_3, X_4)$, in other words, $m = 3$ in the case $k = 2$.
\end{example}

\begin{example}[OU-type model with a periodic component]
  \label{ex:periodic-example}
  The next example we consider is a periodic extension of the stable
  Ornstein--Uhlenbeck process from
  \cref{ex:ornstein-uhlenbeck}.
  Let~$\theta = (\theta_1, \theta_2) \in (0, \infty)^2$ and consider
  the kernel function:
  \begin{equation*}
    g_\theta(u) = \exp(-\theta_1 u - \theta_2 f(u)) \1_{(0, \infty)}(u) , \qquad u \in \bbR,
  \end{equation*}
  where $f : \bbR \to \bbR$ is a bounded measurable function which is
  either non-negative or non-positive and has period $1$, i.e.
  $f(x + 1) = f(x)$ for all $x$. If $f$ does not vanish except on
  Lebesgue null set, then $\theta \mapsto \varphi_{\beta, \theta}^m$
  for $m = 2$ is injective. If, in addition, $f$ is negative then
  Assumption~(B)\cref{it:ass:B:2} is satisfied except possibly at
  $\beta = 1$, and condition~(A)\cref{it:ass:A:4} holds. We refer to
  \cref{sec:example-proofs} for the proof of these statements.
\end{example}

\begin{example}[Modulated OU]
  \label{ex:modulated-OU}
  \noindent Consider the process $X$ defined  at \cref{eq:MA-parametric} with kernel given
  by
  \begin{equation}
    \label{eq:modulated-OU}
    g_\theta(s) = \theta_1 s \exp(- \theta_2 s) \1_{(0, \infty)}(s), \qquad s \in \bbR.
  \end{equation}
  Under the assumptions on the parameters $\theta \in (0, \infty)^2$
  and $\beta \in (1, 2)$ it is possible to prove that $\theta$ is not
  identifiable from $m = 1$ while it is in the case $m =
  2$. Furthermore, condition~(A)\cref{it:ass:A:4} is satisfied. We
  refer to \cref{sec:proof-modulated-OU} for the full exposition of
  these claims.
\end{example}

\begin{example}[CARMA processes]
  \label{ex:CARMA}
  Consider integers $p > q$. The CARMA($p, q$) process
  $(Y_t)_{t \in \bbR}$ with parameters
  $a_1, \ldots, a_p, b_0, \ldots, b_{q - 1} \in \bbR$ driven by $L$ is
  the solution to the stochastic differential equation
  \begin{equation}
    \label{eq:CARMA}
    X_t = b^\top Y_t
    \qquad \text{with} \qquad
    \di Y_t - A Y_t \Di t = e \Di L_t,
  \end{equation}
  where $e$ and $b$ are the $p$-dimensional column vectors given by
  \begin{equation*}
    e = (0, \ldots, 0, 1)^\top \qtq{and}
    b = (b_0, \ldots, b_{p - 1})^\top,
  \end{equation*}
  where $b_q = 1$ and $b_i = 0$ for all $q < i < p$ and $A$ is the
  $p \times p$ matrix given by
  \begin{equation*}
    A =
    \begin{pmatrix}
      0 & 1 & 0 & \cdots & 0
      \\
      0 & 0 & 1 & \cdots & 0
      \\
      \vdots & \vdots & \vdots & \ddots & 1
      \\
      -a_p & -a_{p - 1} & a_{p - 2} & \cdots & -a_1
    \end{pmatrix}.
  \end{equation*}
  CARMA($p, q$) processes fit within the framework of
  \cref{eq:MA-parametric} since if the eigenvalues of $A$ have
  strictly negative real part, then a unique stationary solution of
  \cref{eq:CARMA} exists and is given by
  \begin{equation*}
    X_t = \int_{\bbR} b^\top \e^{A (t - s)} e \1_{[0, \infty)}(t - s) \Di L_s,
    \qquad t \in \bbR,
  \end{equation*}
  see \cite[Proposition~1]{BrocEsti}. In this example we discuss a
  specific three-dimensional sub-class of CARMA($2, 1$) processes,
  which corresponds to the choice $\lambda \coloneqq -\sqrt{a_2}$ and
  $a_1 = 2 \sqrt{a_2} = -2 \lambda$. The parameter of interest becomes
  $\xi = (\beta, b_0, \lambda)$ and we further assume that
  $\beta \in (1,2)$ and $\theta \coloneqq b_0 + \lambda > 0$. In this
  setting the matrix $A$ is given by
  \begin{equation*}
    A = 
    \begin{pmatrix}
      0 & 1
      \\
      -\lambda^2 & 2\lambda
    \end{pmatrix}
  \end{equation*}
  and $\lambda < 0$ is the only eigenvalue of $A$. We thus obtain the
  Jordan normal form
  \begin{equation*}
    A = S
    \begin{pmatrix}
      \lambda & 1
      \\
      0 & \lambda
    \end{pmatrix}
    S^{-1},
    \qquad
    S =
    \begin{pmatrix}
      1 & 0
      \\
      \lambda & 1
    \end{pmatrix}
    ,
    \qquad
    S^{-1} =
    \begin{pmatrix}
      1 & 0
      \\
      - \lambda & 1
    \end{pmatrix}
    .
  \end{equation*}
  Using this representation elementary matrix algebra yields the
  identity
  \begin{equation*}
    g(s)= b^{\top} \exp(sA) e \1_{[0,\infty)}(s)
    = (1 + \theta s) \exp(\lambda s) \1_{[0, \infty)}(s).
  \end{equation*}
  In \cref{sec:proof-CARMA} we show that, under the condition
  $\beta \in (1,2)$, the parameters of the model are identifiable in
  the case $m = 2$ and condition~(A)\cref{it:ass:A:4} holds.
\end{example}

\subsection{Parametric estimation via minimal contrast approach}
\label{sec:estimation}

We note first that the discrete time process $(X_t)_{t \in \bbZ}$ is
ergodic according to \cite{CambErgo}, and so is the sequence
\begin{equation*}
  Y_i = f(X_{i + 1}, \ldots, X_{i + m}), \qquad \text{$i \in \bbZ$,}
\end{equation*}
for any measurable function $f$. Hence, we obtain by Birkhoff's
ergodic theorem the strong consistency (of the real part) of the joint
empirical characteristic function:
\begin{equation}
  \label{phiconv}
  \varphi_n(u_1, \ldots, u_m) = \frac{1}{n} \sum_{i = 0}^{n - m} \cos\Bigl(\sum_{k = 1}^{m} u_k X_{i + k} \Bigr)
  \asTo \Expt[\Big]{\cos\Bigl(\sum_{k = 0}^{m-1} u_k X_{1 + k} \Bigr)}
  = \varphi_{\xi}^m(u_1, \ldots, u_m),
\end{equation}
where $\xi = (\beta, \theta) \in \Xi$ denotes the unknown parameter of
the model. To reduce cumbersome notation we drop the dependence on $m$
in the characteristic function and simply write $\varphi_\xi$ from now
on. For a weight function $w$ introduced in the previous section, we
denote by $F : \cL_w^2(\bbR_+^m) \times \Xi \to \bbR$ the map
\begin{equation*}
  F(\psi, \xi) = \norm{\psi - \varphi_\xi}_{w, 2}^2.
\end{equation*}
The minimal contrast estimator $\xi_n$ of $\xi$ is then defined as
\begin{equation}
  \label{eq:MCE}
  \xi_n \in \argmin_{\xi \in \Xi} F(\varphi_n, \xi)
  = \argmin_{\xi \in \Xi} \int_{\bbR_+^m} (\varphi_n(u) - \varphi_\xi(u))^2 w(u) \Di u,
\end{equation}
and we remark that $\xi_n$ can be chosen universally measurable by
\cite[Theorem~2.17(d)]{StinMeas}. To obtain the asymptotic normality
of the minimal contrast estimator $\xi_n$ we will show a central limit
theorem for the statistic
$\sqrt{n} (\varphi_n(u_1, \ldots, u_m) - \varphi_{\xi}(u_1, \ldots,
u_m))$ using \cref{thm:MA-CLT} and then apply a functional version of
the implicit function theorem. For this purpose we introduce a
centred Gaussian field $(G_u)_{u \in \bbR_+^m}$ whose covariance
kernel is defined as
\begin{equation}
  \label{covG}
  \cov (G_u,G_v) =  \sum_{l \in \bbZ} \cov(\cos(\iprod{u, Z_0}_{\bbR^m}), \cos(\iprod{v, Z_l}_{\bbR^m})),
\end{equation}
where $Z_k = (X_{1 + k}, \ldots, X_{m + k})$. The main theoretical
result of the paper is the strong consistency and asymptotic normality
of the minimal contrast estimator $\xi_n$.

\begin{theorem}
  \label{thm:MCE-conv}
  Let $(\xi_n)$ be the minimal contrast estimator at
  \cref{eq:MA-parametric} associated with the true parameter
  $\xi_0 = (\beta_0, \theta_0)$. Suppose that Assumptions~(A) and~(B)
  hold for the underlying family of kernels $(g_{\xi})_{\xi \in
    \Xi}$. Assume that the weight function $w$ is continuous and
  $\int_{\bbR^m_+} \norm{u}_{\bbR^m}^2 w(u) \Di u < \infty$.
  \begin{tlist}
  \item\label{it:thm:MCE-conv:1} $\xi_n \to \xi_0$ almost surely as
    $n \to \infty$.

  \item\label{it:thm:MCE-conv:2} The convergence as $n \to \infty$
    \begin{equation*}
      \sqrt{n} (\xi_n - \xi_0) \distTo \bigl(\nabla_{\!\xi}^2 F(\varphi_{\xi_0}, \xi_0) \bigr)^{-1} (\iprod{\partial_{\xi}^i \varphi_{\xi_0}, G}_w)_{i = 1, \ldots, d + 1}
    \end{equation*}
    holds, where $G = (G_u)_{u \in \bbR_+^m}$ is a continuous
    zero-mean Gaussian random field with covariance kernel defined by
    \eqref{covG}. In particular, the above limit is a normally
    distributed $(d + 1)$-dimensional random vector.
  \end{tlist}
\end{theorem}

\noindent Direct computation shows that the Hessian
$\nabla^2_{\smash{\!\xi}} F(\varphi_{\xi_0}, \xi_0)$ is a Gramian
matrix with respect to the first order derivatives
$(\partial_\xi^i \varphi_\xi)_{i = 1, \ldots, d + 1}$ and the inner
product $\iprod{\argmrk, \argmrk}_w$ and therefore
Assumption~(A)\cref{it:ass:A:4} ensures that the Hessian is
invertible. In principle, the normal limit in \cref{thm:MCE-conv} is
explicit up to the knowledge of the parameter $\xi_0$, but due to the
complex covariance kernel of the process $G$ it is hard to apply the
central limit theorem to obtain confidence regions. Instead one may
use a parametric bootstrap approach as it has been suggested in
\cite[Section~4.2]{LjunAMin}.

We remark that the convergence rate is $\sqrt{n}$ for all
parameters. Due to the non-Markovian structure of the general model
\cref{eq:MA-parametric} it is a non-trivial task to assess the
optimality of this~rate. As we have discussed in
\cref{ex:ornstein-uhlenbeck} the rate $\sqrt{n}$ can be suboptimal in
the particular case of the drift parameter in an Ornstein--Uhlenbeck
model.

\begin{remark}[Extension to general Lévy drivers] 
  If we drop the requirement for estimation of $\beta$ we can consider
  a larger class of Lévy drivers. Indeed, according to \cite{AzmoMult}
  the statement of \cref{thm:MA-CLT} still holds for a symmetric Lévy
  process $L$, which admits a Lévy density $\nu$ such that
  \begin{equation*}
    \nu(x) \leq C \abs{x}^{-1 - \beta} \quad \text{for all $x \neq 0$.}
  \end{equation*}
  In this case the characteristic
  function takes on a more complicated form. Indeed, by
  \cite[Theorem~2.7]{RajpSpec} it holds that
  \begin{equation*}
    \Expt[\big]{\e^{\I \iprod{u, (X_1, \ldots, X_m)}_{\bbR^m}}}
    = \exp\Bigl(\int_{\bbR} \int_{\bbR} [\cos(\iprod{u, x (g_\xi(z + i))_{i = 0, \ldots, m - 1}}_{\bbR^m}) - 1] \isp \nu(\di x) \Di z \Bigr).
  \end{equation*}
  In principle, the asymptotic theory of \cref{thm:MCE-conv} can be
  extended to this more general setting. However, the proof of the
  asymptotic normality relies on the existence of a continuous
  modification of the random field $(G_u)_{u \in \bbR_+^m}$ and the
  behaviour of $\Expt{G_u^2}$ in $u \in \bbR_+^m$
  (cf.~\cref{sec:limit}), which requires a different treatment
  compared to the $\beta$-stable case.
\end{remark}

\section{A simulation study}
\label{sec3}

\noindent In this section we will demonstrate the finite sample
performance of our estimator for three examples, which are supposed to
highlight different aspects of the minimal contrast approach. First,
we will consider the linear fractional stable motion
(cf. \cref{ex:lfsm}) and use $m = 3$ to estimate the three-dimensional
parameter of the model. The second model is the Ornstein--Uhlenbeck
process considered in \cref{ex:ornstein-uhlenbeck}. We will examine
the performance for the full model and also for a known scale
parameter $\sigma$ to assess the improvement of the estimation
procedure.  In the latter setting both $m = 1$ and $m = 2$ are used to
estimate the drift $\lambda$ and the stability index $\beta$, and the
aim of the numerical simulation is to test how the choice of higher
index $m$ affects the performance of the estimator.  The third example
is the \emph{generalized modulated OU-process}, which has not been
shown to satisfy the main assumptions of the paper. We will use $m=2$
to estimate the three-dimensional parametric model and test how our
method works in this framework.

Since the weight function $w$ depends on $m$ implicitly via
its domain we need a function, which is reasonably compatible between
different dimensions and we consider therefore throughout this study
the $m$-dimensional Gaussian density with zero mean and a scaled unit
covariance matrix $\nu^2 I_m$:
\begin{equation}
  \label{eq:gauss-weight}
  w_\nu(u) = (2 \pi \nu^2)^{-m/2} \exp \Bigl(-\frac{\norm{u}_{\bbR^m}^2}{2 \nu^2} \Bigr), \qquad u \in \bbR^m, \quad \nu > 0.
\end{equation}
The choice of $\nu$ varies between the three example processes and it
is a subject for future research to automatically determine an optimal
weight. For the computation of the weighted integral in \cref{eq:MCE}
we use Gauss--Laguerre quadrature which is a weighted sum of function
values and the number of weights will also vary depending on the
process.

We note additionally that the minimization involved in computing the
minimal contrast estimator at \cref{eq:MCE} has to be done numerically
and for this we use the method of \cite{NeldASim}, which requires
picking a starting point which naturally will depend on the example
kernel at~hand. Lastly, we remark that the $\beta$-norm of the kernel
function is generally not known explicitly, hence the theoretical
characteristic function is approximated as well.

All tables in this section are based on at least 200 Monte Carlo
repetitions.

\subsection{Linear fractional stable motion}
\label{sec:simulation-lfsm}

Recall from the discussion in \cref{ex:lfsm} that it is prudent to
take higher order increments, and we fix throughout $k = 2$. Moreover,
to properly identify the parameters we consider the characteristic
function of the three-dimensional joint distribution, hence $m =
3$. Next~we consider throughout the weight function at
\cref{eq:gauss-weight} with standard deviation $\nu = 10$ and the
weighted integral is approximated with $12^3 = 1728$ number of
weights. The starting point for the minimization algorithm is
$(\beta, H, \sigma) = (1.5, 0.5, 2)$.

The estimator is tested in the continuous case, so only parameter
combinations resulting in the equality $H - 1/\beta > 0$ are
considered. \Cref{tab:lfsm-1000} reports the bias and standard
deviation in the case of $n = 1000$ for different parameter
combinations, while \cref{tab:lfsm-10000} explores the case
$n = \num{10000}$. We observe a rather good performance of all
estimators with superior results in the setting $n = \num{10000}$ as
expected from our theoretical statements. We note that the estimator
of the scale parameter $\sigma$ performs the best, which is in line
with earlier findings of \cite{MazuEsti}.

\begin{table}[ht]
  \caption{Absolute value of bias (|Bias|) and standard deviation
    (Std) for $n = \num{1000}$ and $\sigma = 0.3$ for the linear fractional stable motion}
  \label{tab:lfsm-1000}
  \centering
  \begin{tabular}{l l *{6}{S}}
    & & \multicolumn{3}{c}{|Bias|} & \multicolumn{3}{c}{Std}
    \\
    \cmidrule(lr){3-5}
    \cmidrule(lr){6-8}
    $H$ & $\beta$ & $\widehat{\beta}_n$ & $\widehat{H}_n$ & $\widehat{\sigma}_n$ & $\widehat{\beta}_n$ & $\widehat{H}_n$ & $\widehat{\sigma}_n$
    \\
    \midrule
    0.6 & 1.8 & 0.01755197 & 0.04776433 & 0.03617623 & 0.19502015 & 0.27298646 & 0.07047406
    \\
    \midrule
    0.7 & 1.6 & 0.07050554 & 0.17096827 & 0.08344025 & 0.24087475 & 0.35813020 & 0.09474459
    \\
    & 1.8 & 0.010619245 & 0.004120555 & 0.011984553 & 0.1765943 & 0.2093519 & 0.0428728
    \\
    \midrule
    0.8 & 1.4 & 0.08617515 & 0.24435164 & 0.08619785 & 0.2348498 & 0.3456916 & 0.1043623
    \\
    & 1.6 & 0.02497482 & 0.05966458 & 0.02698820 & 0.17827942 & 0.24658854 & 0.05405855
    \\
    & 1.8 & 0.011965440 & 0.006017185 & 0.004421216 & 0.14524688 & 0.15777800 & 0.02874621
    \\
    \bottomrule
  \end{tabular}
\end{table}

\begin{table}[ht]
  \caption{Absolute value of bias (|Bias|) and standard deviation
    (Std) for $n = \num{10000}$ and $\sigma = 0.3$ for the linear
    fractional stable motion}
  \label{tab:lfsm-10000}
  \centering
  \begin{tabular}{l l *{6}{S}}
    & & \multicolumn{3}{c}{|Bias|} & \multicolumn{3}{c}{Std}
    \\
    \cmidrule(lr){3-5}
    \cmidrule(lr){6-8}
    $H$ & $\beta$ & $\widehat{\beta}_n$ & $\widehat{H}_n$ & $\widehat{\sigma}_n$ & $\widehat{\beta}_n$ & $\widehat{H}_n$ & $\widehat{\sigma}_n$
    \\
    \midrule
    0.6 & 1.8 & 0.01329346 & 0.04559194 & 0.02544133 & 0.12717678 & 0.20072939 & 0.05317261
    \\
    \midrule
    0.7 & 1.6 & 0.02377204 & 0.08178138 & 0.03306955 & 0.10050421 & 0.21469962 & 0.06847038
    \\
    & 1.8 & 0.005960964 & 0.014658876 & 0.006613568 & 0.08691589 & 0.11526434 & 0.01733413
    \\
    \midrule
    0.8 & 1.4 & 0.03469922 & 0.15364074 & 0.05036789 & 0.10948992 & 0.25458024 & 0.08652012
    \\
    & 1.6 & 0.0077620137 & 0.0052886243 & 0.0008365061 & 0.066512840 & 0.084294208 & 0.008512437
    \\
    & 1.8 & 0.0031975759 & 0.0020427964 & 0.0009052234 & 0.059715064 & 0.073210702 & 0.006666248
    \\
    \bottomrule
  \end{tabular}
\end{table}

\subsection{Ornstein--Uhlenbeck}
\label{sec:OU-sim}

\noindent In this subsection we consider the Ornstein--Uhlenbeck
kernel from \cref{ex:ornstein-uhlenbeck}. We start with a two-parameter
submodel, where $\sigma = 1$ is fixed.
In this case Assumption~(A) is satisfied for both $m = 2$
and $m = 1$, and we will compare the performance for each of these
dimensions. Akin to \cref{sec:gen-OU-sim} we pick $20^m$, $m=1,2$,
number of weights in the integral approximation with weight function
chosen as in \cref{eq:gauss-weight} with $\nu = 1$. The starting point
for the minimization algorithm is throughout
$(\beta, \lambda) = (1.5, 0.5)$.

Tables~\ref{tab:OU-1} and~\ref{tab:OU-2} demonstrate the simulation
results for $m = 1$ and $m = 2$, respectively. We observe a rather
convincing performance for both estimators in all settings, but the
choice $m = 1$ clearly outperforms the setting $m=2$. We conjecture
that it has a theoretical background, i.e. the asymptotic variances in
\cref{thm:MCE-conv}\cref{it:thm:MCE-conv:2} are smaller for $m = 1$,
and a numerical background. Indeed, the minimization algorithm has a
worse performance for higher values of $m$. For this reason it is
advisable to use the minimal $m$, which identifies the parameters of
the model.

\begin{table}[ht!]
  \footnotesize
  \caption{Absolute value of bias (|Bias|) and standard deviation
    (Std) for $m = 1$ and $n \in \Set{10^3, 10^4}$}
  \label{tab:OU-1}
  \begin{minipage}{0.485\linewidth}
    \begin{tabular}{*{2}{l} *{4}{S}}
      \multicolumn{2}{c}{$n = \num{1000}$} & \multicolumn{2}{c}{|Bias|} & \multicolumn{2}{c}{Std}
      \\
      \cmidrule(lr){3-4}
      \cmidrule(lr){5-6}
      $\beta$ & $\lambda$ & $\widehat{\beta}_n$ & $\widehat{\lambda}_n$ & $\widehat{\beta}_n$ & $\widehat{\lambda}_n$
      \\
      \midrule
      1.2 & 0.25 & 0.01848909 & 0.00298011 & 0.10724035 & 0.05435007
      \\
                                           & 0.75 & 0.014410230 & 0.006109171 & 0.06265485 & 0.06828288
      \\
                                           & 1 & 0.010672989 & 0.001802275 & 0.05731801 & 0.07553548
      \\
                                           & 1.25 & 0.008445935 & 0.006161037 & 0.05311887 & 0.08564561
      \\
                                           & 1.5 & 0.013511102 & 0.005816892 & 0.05612671 & 0.09014910
      \\
                                           & 2 & 0.004374805 & 0.002762616 & 0.05434291 & 0.12821467
      \\
                                           & 2.5 & 0.01216144 & 0.01498229 & 0.05832045 & 0.15299989
      \\
      \midrule
      1.4 & 0.25 & 0.009714624 & 0.007854220 & 0.12137080 & 0.05300858
      \\
                                           & 0.75 & 0.0047482304339006 & 0.00287290473094926 & 0.0661387572426329 & 0.0669293087852905
      \\
                                           & 1 & 0.00361539739449324 & 0.00927334360003162 & 0.059267526221943 & 0.0645653065551836
      \\
                                           & 1.25 & 0.00424308274144658 & 0.00177689666576808 & 0.0571739502894545 & 0.0757308436943691
      \\
                                           & 1.5 & 0.013837199413433 & 0.00229557972986538 & 0.0550414056292689 & 0.0826170497624134
      \\
                                           & 2 & 0.00913624451278139 & 0.00393175049424133 & 0.0595093295986336 & 0.107195159849191
      \\
                                           & 2.5 & 0.00598827179185002 & 0.00720819822647112 & 0.0608483579869255 & 0.150680150388701
      \\
      \midrule
      1.6 & 0.25 & 0.00991939363327199 & 0.00116601486397205 & 0.11426471352373 & 0.0512797249824008
      \\
                                           & 0.75 & 0.007062397 & 0.006647522 & 0.06044286 & 0.06044254
      \\
                                           & 1 & 0.0075872602327125 & 0.00157880585951686 & 0.0589541415350359 & 0.0668787507503899
      \\
                                           & 1.25 & 0.0116218358861453 & 0.0041717146808975 & 0.0533326842686992 & 0.0758653351070753
      \\
                                           & 1.5 & 0.00196232149456654 & 0.00394197147749709 & 0.0563245124305967 & 0.0781211107476137
      \\
                                           & 2 & 0.0100698034974203 & 0.00738438796732543 & 0.054039376049033 & 0.10205742725948
      \\
                                           & 2.5 & 0.0144099569116687 & 0.00612108068874262 & 0.0566960526110582 & 0.12832738632701
      \\
      \midrule
      1.8 & 0.25 & 0.0106405089611421 & 0.000444169686732526 & 0.101316516427916 & 0.0417299598567335
      \\
                                           & 0.75 & 0.0110582643655117 & 0.00070496631112138 & 0.0586348947732862 & 0.0597074164339262
      \\
                                           & 1 & 0.00208424180495581 & 0.000676191973232099 & 0.0528626909263887 & 0.0649203474453415
      \\
                                           & 1.25 & 0.00880894732218529 & 0.00425489217976094 & 0.0452888157170556 & 0.0764369405038173
      \\
                                           & 1.5 & 0.00916397525184931 & 0.0136345138631531 & 0.0493683558026885 & 0.0825031586001826
      \\
                                           & 2 & 0.00838810148086022 & 0.00247983851069766 & 0.0480998255083872 & 0.101452652351612
      \\
                                           & 2.5 & 0.0143668493459415 & 0.00448923044438443 & 0.0446110145719369 & 0.127276181115318
      \\
      \bottomrule
    \end{tabular}
  \end{minipage}%
  \hfill%
  \begin{minipage}{0.485\linewidth}
    \begin{tabular}{*{2}{l} *{4}{S}}
      \multicolumn{2}{c}{$n = \num{10000}$} & \multicolumn{2}{c}{|Bias|} & \multicolumn{2}{c}{Std}
      \\
      \cmidrule(lr){3-4}
      \cmidrule(lr){5-6}
      $\beta$ & $\lambda$ & $\widehat{\beta}_n$ & $\widehat{\lambda}_n$ & $\widehat{\beta}_n$ & $\widehat{\lambda}_n$
      \\
      \midrule
      1.2 & 0.25 & 0.0015682987 & 0.0003714286 & 0.03212787 & 0.01736130 
      \\
                                            & 0.75 & 0.003880790 & 0.002776587 & 0.01867913 & 0.01990753 
      \\
                                            & 1 & 0.0007798106 & 0.00002125211 & 0.01971349 & 0.02436773
      \\
                                            & 1.25 & 0.000007390473 & 0.002838348248 & 0.01885158 & 0.02648644
      \\
                                            & 1.5 & 0.0004131053 & 0.0033256170 & 0.01637904 & 0.03233758
      \\
                                            & 2 & 0.0002595931 & 0.0029946123 & 0.01609831 & 0.03652805
      \\
                                            & 2.5 & 0.001619217 & 0.008962453 & 0.01794475 & 0.04664228
      \\
      \midrule
      1.4 & 0.25 & 0.003534414 & 0.001370142 & 0.03770344 & 0.01718855 
      \\
                                            & 0.75 & 0.0016623841 & 0.0009810154 & 0.02024432 & 0.01943037 
      \\
                                            & 1 & 0.001691820 & 0.002218541 & 0.01963434 & 0.02410339 
      \\
                                            & 1.25 & 0.001625175 & 0.002642474 & 0.01740864 & 0.02742331
      \\
                                            & 1.5 & 0.000003488464 & 0.009235335522 & 0.01768938 & 0.02814869
      \\
                                            & 2 & 0.001586205 & 0.006883170 & 0.01661806 & 0.03624477
      \\
                                            & 2.5 & 0.00137793 & 0.01553322 & 0.01940200 & 0.04038705
      \\
      \midrule
      1.6 & 0.25 & 0.007866633 & 0.001880583 & 0.04388563 & 0.01692359
      \\
                                            & 0.75 & 0.002183814 & 0.001406974 & 0.01836295 & 0.01703971
      \\
                                            & 1 & 0.0007701028 & 0.0020470710 & 0.01817529 & 0.02124010
      \\
                                            & 1.25 & 0.001490512 & 0.003229174 & 0.01830536 & 0.02390481
      \\
                                            & 1.5 & 0.0003078567 & 0.0056215123 & 0.01688709 & 0.02706915
      \\
                                            & 2 & 0.0005057938 & 0.0132906524 & 0.01658700 & 0.03254491
      \\
                                            & 2.5 & 0.001960681 & 0.019155041 & 0.01780907 & 0.04009728
      \\
      \midrule
      1.8 & 0.25 & 0.001507192 & 0.001103222 & 0.03922212 & 0.01518641
      \\
                                            & 0.75 & 0.001346943 & 0.001046849 & 0.01867697 & 0.01761996
      \\
                                            & 1 & 0.001877630 & 0.005138061 & 0.01618150 & 0.01895479
      \\
                                            & 1.25 & 0.002014258 & 0.006703163 & 0.01585133 & 0.02317175
      \\
                                            & 1.5 & 0.001189804 & 0.011284270 & 0.01505653 & 0.02627517
      \\
                                            & 2 & 0.003151372 & 0.015856142 & 0.01464786 & 0.02984909
      \\
                                            & 2.5 & 0.000009015827 & 0.025928007405 & 0.01434226 & 0.04105065
      \\
      \bottomrule
    \end{tabular}
  \end{minipage}
\end{table}

\begin{table}[htpb!]
  \footnotesize
  \caption{Absolute value of bias (|Bias|) and standard deviation
    (Std) for $m = 2$ and $n \in \Set{10^3, 10^4}$}
  \label{tab:OU-2}
  \begin{minipage}{0.485\linewidth}
    \begin{tabular}{*{2}{l} *{4}{S}}
      \multicolumn{2}{c}{$n = \num{1000}$} & \multicolumn{2}{c}{|Bias|} & \multicolumn{2}{c}{Std}
      \\
      \cmidrule(lr){3-4}
      \cmidrule(lr){5-6}
      $\beta$ & $\lambda$ & $\widehat{\beta}_n$ & $\widehat{\lambda}_n$ & $\widehat{\beta}_n$ & $\widehat{\lambda}_n$
      \\
      \midrule
      1.2 & 0.25 & 0.398830316711305 & 0.111326038194209 & 0.104414969698091 & 0.0622615237713748
      \\
                                           & 0.75 & 0.0666188107821528 & 0.0362738683345553 & 0.217827814603489 & 0.193161803787735
      \\
                                           & 1 & 0.0149518611088542 & 0.00811045364726026 & 0.0928473528932355 & 0.133901648603892
      \\
                                           & 1.25 & 0.0124994187863381 & 0.0104395300318076 & 0.0680962447334966 & 0.12262844067584
      \\
                                           & 1.5 & 0.00536259655729432 & 0.00627437821881238 & 0.0626046277753657 & 0.118103055322615
      \\
                                           & 2 & 0.00729758257900981 & 0.0090450511115594 & 0.0646339211675458 & 0.147174301241766
      \\
                                           & 2.5 & 0.0612553213380231 & 0.147746387233559 & 0.0635658942107653 & 0.0626734911801756
      \\
      \midrule
      1.4 & 0.25 & 0.202821771709335 & 0.0530979485176273 & 0.140696548116578 & 0.106147599553645
      \\
                                           & 0.75 & 0.0483852289952722 & 0.0203830359031427 & 0.179292893341633 & 0.162140813559506
      \\
                                           & 1 & 0.00634396818571181 & 0.00631839261939104 & 0.0847863297900298 & 0.116539795987243
      \\
                                           & 1.25 & 0.0124139994278512 & 0.00667614891437451 & 0.0713533986413619 & 0.109620633760984
      \\
                                           & 1.5 & 0.00252551089980191 & 0.00667837780232072 & 0.0720939246017064 & 0.120413306610863
      \\
                                           & 2 & 0.00796514654151959 & 0.0202701570102952 & 0.0571961100295115 & 0.126904632697804
      \\
                                           & 2.5 & 0.0592519425754683 & 0.139491203155629 & 0.0733736779064118 & 0.0481818892517774
      \\
      \midrule
      1.6 & 0.25 & 0.11201656459104 & 0.107808332609746 & 0.300867495073656 & 0.213912481264616
      \\
                                           & 0.75 & 0.0481076641826932 & 0.021037806192406 & 0.166916684816993 & 0.160168650305801
      \\
                                           & 1 & 0.0165298470103221 & 0.0159204885628286 & 0.0909406768682754 & 0.116404188834423
      \\
                                           & 1.25 & 0.00724971086044701 & 0.00170349375832557 & 0.0666329098487231 & 0.103894633929661
      \\
                                           & 1.5 & 0.00124059527004872 & 0.00784764211516409 & 0.0666649456314748 & 0.0989651756752001
      \\
                                           & 2 & 0.00370464505350698 & 0.0132544140649173 & 0.0687797042738714 & 0.117148550843464
      \\
                                           & 2.5 & 0.0873222119910813 & 0.136431290125287 & 0.0849544599179699 & 0.0431490694776375
      \\
      \midrule
      1.8 & 0.25 & 0.2478053935789 & 0.175104078295621 & 0.35837638125669 & 0.223228792110426
      \\
                                           & 0.75 & 0.0194247665845284 & 0.00147357152518457 & 0.118245005849628 & 0.125309270872433
      \\
                                           & 1 & 0.0112118886812687 & 0.000676896612096245 & 0.0755273485961436 & 0.101043231334856
      \\
                                           & 1.25 & 0.00979658139357387 & 0.00833484421073694 & 0.0586919044901048 & 0.0881463751877903
      \\
                                           & 1.5 & 0.0149907428598803 & 0.00203118587409179 & 0.0539521046748819 & 0.0972707159155239
      \\
                                           & 2 & 0.0187396327775884 & 0.0120707035443943 & 0.0632344951968781 & 0.110601830320216
      \\
                                           & 2.5 & 0.09479743353324 & 0.134573882792598 & 0.0802309760751146 & 0.050162618889132
      \\
      \bottomrule
    \end{tabular}
  \end{minipage}%
  \hfill%
  \begin{minipage}{0.485\linewidth}
    \begin{tabular}{*{2}{l} *{4}{S}}
      \multicolumn{2}{c}{$n = \num{10000}$} & \multicolumn{2}{c}{|Bias|} & \multicolumn{2}{c}{Std}
      \\
      \cmidrule(lr){3-4}
      \cmidrule(lr){5-6}
      $\beta$ & $\lambda$ & $\widehat{\beta}_n$ & $\widehat{\lambda}_n$ & $\widehat{\beta}_n$ & $\widehat{\lambda}_n$
      \\
      \midrule
      1.2 & 0.25 & 0.3921843 & 0.1142326 & 0.046388553 & 0.009796207
      \\
                                            & 0.75 & 0.0004914377 & 0.0012899577 & 0.03905227 & 0.04800461
      \\
                                            & 1 & 0.001857013 & 0.001237152 & 0.0259636 & 0.0421193
      \\
                                            & 1.25 & 0.0007196836 & 0.0023631634 & 0.02203777 & 0.03767779
      \\
                                            & 1.5 & 0.0002974257 & 0.0030344414 & 0.02178713 & 0.04284708
      \\
                                            & 2 & 0.0004709638 & 0.0050433668 & 0.01951663 & 0.04536475
      \\
                                            & 2.5 & 0.03869304 & 0.11855428 & 0.02040855 & 0.00241842
      \\
      \midrule
      1.4 & 0.25 & 0.19159850 & 0.05989281 & 0.07306659 & 0.04020624
      \\
                                            & 0.75 & 0.002354319 & 0.001543543 & 0.04385633 & 0.04998472
      \\
                                            & 1 & 0.001871446 & 0.002410993 & 0.02572618 & 0.03625083
      \\
                                            & 1.25 & 0.0009098926 & 0.0001940331 & 0.0234778 & 0.0360772
      \\
                                            & 1.5 & 0.0011508784 & 0.0004452605 & 0.02111279 & 0.03807391
      \\
                                            & 2 & 0.002745685 & 0.002011239 & 0.02266111 & 0.03967609
      \\
                                            & 2.5 & 0.05046195 & 0.11838351 & 0.0243336319 & 0.0006158415
      \\
      \midrule
      1.6 & 0.25 & 0.005144016 & 0.013832271 & 0.1794123 & 0.1027737
      \\
                                            & 0.75 & 0.008394238 & 0.006100152 & 0.04514113 & 0.04789958
      \\
                                            & 1 & 0.0002332647 & 0.0023374798 & 0.02529702 & 0.03243626
      \\
                                            & 1.25 & 0.0003219736 & 0.0047006103 & 0.02056957 & 0.03054649
      \\
                                            & 1.5 & 0.0003378693 & 0.0040023939 & 0.02003564 & 0.03336524
      \\
                                            & 2 & 0.001486581 & 0.005952279 & 0.02004142 & 0.03890433
      \\
                                            & 2.5 & 0.06040934 & 0.11845340 & 0.028693724 & 0.001479964
      \\
      \midrule
      1.8 & 0.25 & 0.2108622 & 0.1160020 & 0.2539452 & 0.1351242
      \\
                                            & 0.75 & 0.001606988 & 0.002273714 & 0.03894069 & 0.03946084
      \\
                                            & 1 & 0.000126265 & 0.002489021 & 0.02429793 & 0.03158331
      \\
                                            & 1.25 & 0.0003520846 & 0.0035807409 & 0.01783747 & 0.02662721
      \\
                                            & 1.5 & 0.0001025649 & 0.0042289669 & 0.01728060 & 0.02802578
      \\
                                            & 2 & 0.001233173 & 0.009184195 & 0.01806001 & 0.03706257
      \\
                                            & 2.5 & 0.08013119 & 0.11840970 & 0.0343235417 & 0.0009741405
      \\
      \bottomrule
    \end{tabular}
  \end{minipage}
\end{table}

We now consider the full Ornstein--Uhlenbeck model from
\cref{ex:ornstein-uhlenbeck} with parameters $\beta\in (0,2)$,
$\lambda>0$ and a non-fixed scale $\sigma > 0$. For comparison with
the case of fixed scale we consider as starting point
$(\beta, \lambda, \sigma) = (1.5, 0.5, 1.1)$ for the minimization
algorithm. To~avoid innumerable parameter combinations we consider
only $\beta \in \Set{1.4, 1.6}$, $\lambda \in \Set{0.25, 0.75}$ and
$\sigma \in \Set{0.9, 1}$. Lastly, we have $\nu = 1$ as in the
previous simulation.

\begin{table}[htpb!]
  \footnotesize
  \caption{Absolute value of bias (|Bias|) and standard deviation
    (Std) for $m = 2$ and $n = \num{10000}$}
  \label{tab:OU-gen-sim}
  \centering
  \begin{tabular}{*{3}{l} *{6}{S}}
    & & & \multicolumn{3}{c}{|Bias|} & \multicolumn{3}{c}{Std}
    \\
    \cmidrule(lr){4-6}
    \cmidrule(lr){7-9}
    $\beta$ & $\lambda$ & $\sigma$ & \multicolumn{1}{c}{$\widehat{\beta}_n$} & \multicolumn{1}{c}{$\widehat{\lambda}_n$} & \multicolumn{1}{c}{$\widehat{\sigma}_n$} & \multicolumn{1}{c}{$\widehat{\beta}_n$}& \multicolumn{1}{c}{$\widehat{\lambda}_n$} & \multicolumn{1}{c}{$\widehat{\sigma}_n$}
    \\
    \midrule
    1.4 & 0.25 & 0.9 & 0.0095413909141413 & 0.194649992241076 & 0.404108409385271 & 0.197145362541796 & 0.270795600150288 & 0.399829744696153
    \\
    1.4 & 0.25 & 1 & 0.0606339095891109 & 0.0735730853114305 & 0.280088318379048 & 0.0101388449303199 & 0.0157200747280256 & 0.00788647500987636
    \\
    \midrule
    1.4 & 0.75 & 0.9 & 0.00129109073482492 & 0.0512393180789805 & 0.044204012367504 & 0.0304664290398833 & 0.0505733685099253 & 0.0410528672561375
    \\
    1.4 & 0.75 & 1 & 0.00428372805126886 & 0.040975504074141 & 0.0443793996744015 & 0.0450798075459594 & 0.0637411553326562 & 0.0444046663734946
    \\
    \midrule
    1.6 & 0.25 & 0.9 & 0.0527050621511536 & 0.21363762407373 & 0.316706861632831 & 0.270144158371462 & 0.185193363909624 & 0.199750726259358
    \\
    1.6 & 0.25 & 1 & 0.156874700931145 & 0.208368116181837 & 0.240081744723229 & 0.237989012284211 & 0.224222907110822 & 0.290384816364956
    \\
    \midrule
    1.6 & 0.75 & 0.9 & 0.00105218435182475 & 0.00274998863094855 & 0.00446783480569701 & 0.0300447134140977 & 0.0384430546654956 & 0.0379326684569673
    \\
    1.6 & 0.75 & 1 & 0.002645728316373 & 0.0480165668038269 & 0.041545898280139 & 0.0428490590020055 & 0.030912887417959 & 0.020485390005927
    \\
    \bottomrule
  \end{tabular}
\end{table}

Comparing \cref{tab:OU-gen-sim} to \cref{tab:OU-2} we see that a
fixed, known $\sigma$ significantly increases the performance of the
estimator especially when the starting point for, e.g. $\sigma$ is
further away. Nevertheless, the estimation results in
\cref{tab:OU-gen-sim} are still quite reliable for most parameter
settings.

\subsection{Generalized modulated OU process}
\label{sec:gen-OU-sim}

The generalized modulated OU process is defined via equation
\cref{eq:MA-parametric} with kernel function
\begin{equation*}
  g_{\theta}(s) = s^\sigma \exp(-\lambda s) \1_{(0, \infty)}(s), \qquad s \in \bbR,
\end{equation*}
where $\theta = (\sigma, \lambda) \in (0, \infty)^2$. This class of
kernels has not been shown to satisfy the main assumption of the
paper, but it is easily seen that $m = 1$ is not enough to identify
the parameters in $\theta$. We take $m = 2$ and set the number of
weights to $20$, hence the weighted integral approximation is based on
$20^2 = 400$ nodes. Moreover, the weight function is as in
\cref{eq:gauss-weight} with $\nu = 0.1$. Lastly, we pick as starting
point for the minimization algorithm
$(\beta, \lambda, \sigma) = (1.5, 1, 1)$.

\Cref{tab:mod-OU-05,tab:mod-OU-2} report the finite sample performance
of the estimators for $n = \num{10000}$, and $\sigma = 0.5$ and
$\sigma = 2$, respectively. We observe a good performance of the
estimator $\widehat{\beta}_n$ and a very unsatisfactory performance of
the estimator $\widehat{\sigma}_n$. We conjecture that the reason for
the suboptimal performance lies in the choice of the weight function
$w$, which may have opposite effects on different parameters of the
model, as well as in the minimization algorithm, since it has a
tendency to get stuck in local minima.

\begin{table}[ht]
  \caption{Absolute value of bias (|Bias|) and standard deviation
    for $n = \num{10000}$ and $\sigma = 0.5$ for the generalized
    modulated OU kernel}
  \label{tab:mod-OU-05}
  \centering
  \begin{tabular}{l l *{6}{S}}
    & & \multicolumn{3}{c}{|Bias|} & \multicolumn{3}{c}{Std}
    \\
    \cmidrule(lr){3-5}
    \cmidrule(lr){6-8}
    $\beta$ & $\lambda$ & $\widehat{\beta}_n$ & $\widehat{\lambda}_n$ & $\widehat{\sigma}_n$ & $\widehat{\beta}_n$ & $\widehat{\lambda}_n$ & $\widehat{\sigma}_n$
    \\
    \midrule
    1.8 & 0.5 & 0.01109752 & 0.15848043 & 0.59822835 & 0.04604744 & 0.04441130 & 0.13533531
    \\
    & 0.75 & 0.01960873 & 0.09248290 & 0.56204938 & 0.05418306 & 0.04938010 & 0.17175356
    \\
    & 1.25 & 0.014718235 & 0.006410635 & 0.067124237 & 0.08134501 & 0.11519645 & 0.09459575
    \\
    & 1.5 & 0.002932964 & 0.036121460 & 0.096883978 & 0.08562085 & 0.10055721 & 0.17283024
    \\
    \midrule
    1.2 & 0.5 & 0.006210926 & 0.188094896 & 0.696666524 & 0.03494669 & 0.07322044 & 0.24145869
    \\
    & 0.75 & 0.004444959 & 0.178696908 & 0.808751130 & 0.04396596 & 0.04434375 & 0.04856469
    \\
    & 1.25 & 0.010335131 & 0.008946981 & 0.612374904 & 0.04679042 & 0.05941702 & 0.13065216
    \\
    & 1.5 & 0.01097045 & 0.08860120 & 0.58690032 & 0.05192269 & 0.09514965 & 0.21149758
    \\
    \bottomrule
  \end{tabular}
\end{table}

\begin{table}[ht]
  \caption{Absolute value of bias (|Bias|) and standard deviation
    (Std) for $n = \num{10000}$ and $\sigma = 2$ for the generalized
    modulated OU kernel}
  \label{tab:mod-OU-2}
  \centering
  \begin{tabular}{l l *{6}{S}}
    & & \multicolumn{3}{c}{|Bias|} & \multicolumn{3}{c}{Std}
    \\
    \cmidrule(lr){3-5}
    \cmidrule(lr){6-8}
    $\beta$ & $\lambda$ & $\widehat{\beta}_n$ & $\widehat{\lambda}_n$ & $\widehat{\sigma}_n$ & $\widehat{\beta}_n$ & $\widehat{\lambda}_n$ & $\widehat{\sigma}_n$
    \\
    \midrule
    1.8 & 0.5 & 0.007610437 & 0.031373960 & 0.145797235 & 0.1730268 & 0.2051865 & 0.7289414
    \\
    & 0.75 & 0.00279739 & 0.20887866 & 0.65147455 & 0.03091808 & 0.02413394 & 0.07290650
    \\
    & 1.25 & 0.03137985 & 0.25207178 & 1.32728912 & 0.07269532 & 0.06410093 & 0.12444681
    \\
    & 1.5 & 0.062620147 & 0.006574832 & 1.314710523 & 0.08887894 & 0.10854958 & 0.17250641
    \\
    \midrule
    1.2 & 0.5 & 0.01646307 & 0.02204524 & 0.15310884 & 0.2723836 & 0.1922562 & 0.6672994
    \\
    & 0.75 & 0.001060276 & 0.206548948 & 0.679326158 & 0.03349057 & 0.05210383 & 0.16111681
    \\
    & 1.25 & 0.003685735 & 0.206800080 & 0.768501583 & 0.04735643 & 0.04541476 & 0.03623216
    \\
    & 1.5 & 0.001943202 & 0.172040298 & 1.017552169 & 0.06346027 & 0.09950631 & 0.19276847
    \\
    \bottomrule
  \end{tabular}
\end{table}

\appendix

\section{Proofs}

\label{sec4}

In this section $C > 0$ denotes a generic constant, which may change
from line to line. Recall moreover the shorthand
$\xi = (\beta, \theta)$ for the joint parameters.

\subsection{Proof of Theorem~\ref{thm:MCE-conv}}

This section is devoted to the proof of \cref{thm:MCE-conv}, which is
divided into three steps. In \cref{sec:limit} we analyse the
smoothness properties of the limiting Gaussian field
$(G_u)_{u \in \smash{\bbR_+^m}}$. \Cref{sec4.1.2} presents a general
weak convergence statement for integrals of stochastic
processes. Finally, \cref{sec4.1.3} demonstrates proofs of the
convergence results in \cref{thm:MCE-conv}.

\subsubsection{The limiting Gaussian field}
\label{sec:limit}

To characterize the covariance of the asymptotic Gaussian field
$(G_u)_{u \in \smash{\bbR_+^m}}$ we define a dependence measure
between two $m$-dimensional stable vectors
$Y = (\int h_1 \Di L, \ldots, \int h_m \Di L)$ and
$Z = (\int g_1 \Di L, \ldots, \int g_m \Di L)$:
\begin{equation*}
  U_{Y, Z}(u, v) \coloneqq \Expt[\big]{\e^{\I \iprod{(u, v), (Y, Z)}_{\bbR^{2m}} } }
  - \Expt[\big]{\e^{\I \iprod{u, Y}_{\bbR^m}}} \Expt[\big]{\e^{\I \iprod{v, Z}_{\bbR^m}}},
  \qquad u, v \in \bbR^m.
\end{equation*}
This is a straightforward multivariate extension of the measure
defined in \cite{PipiBoun}. We now apply \cref{thm:MA-CLT} in
conjunction with the smooth and bounded functions
\begin{equation*}
  f_u(x) = \cos(\iprod{u, x}_{\bbR^m}),  \qquad u, x \in \bbR^m,
\end{equation*}
such that we obtain the finite dimensional convergence of the
processes:
\begin{equation}
  \label{eq:fidi-conv-char-functions}
  \sqrt{n}(\varphi_n(u) - \varphi_{\xi}(u))_{u \in \bbR_+^m} \fidiTo[n \to \infty] (G_u)_{u \in \bbR_+^m}.
\end{equation}
Let $Z_0 = (X_1, \ldots, X_m)$ and
$Z_{\ell} = (X_{1 + \ell}, \ldots, X_{m + \ell})$, then the covariance
function $R : \bbR^m \times \bbR^m \to \bbR$ of $G$ is,
cf. \cref{eq:asymp-cov}, given by
\begin{equation*}
  R(u, v) = \sum_{\ell \in \bbZ} r_\ell(u, v),
\end{equation*}
where for $\ell \in \bbZ$
\begin{equation*}
  r_{\ell}(u, v) = \cov(\cos(\iprod{u, Z_0}), \cos(\iprod{v, Z_\ell})), \qquad u, v \in \bbR^m.
\end{equation*}
We will now prove that there exists a version of $G$, which is locally
Hölder continuous up to any order less than $\beta/4$. By Kolmogorov's
criteria and Gaussianity it is enough to prove that for any $T > 0$
there exists a constant $C_T \geq 0$ such that
\begin{equation}
  \label{eq:holder-cont}
  \Expt[\big]{(G_u - G_v)^2} \leq C_T \norm{u - v}^{\beta/2} \qquad \text{for all $u, v \in [0, T]^m$,}
\end{equation}
where $\norm{u - v} = \sum_{i = 1}^{m} \abs{u_i - v_i}$ denotes the
$\ell_1$-norm throughout the rest of this paper. To prove
\cref{eq:holder-cont} note the decomposition
\begin{equation*}
  \Expt[\big]{(G_u - G_v)^2} = R(u, u) - R(u, v) + R(v, v) - R(u, v).
\end{equation*}
Hence by symmetry it suffices to consider the term
\begin{equation*}
  R(u, u) - R(u, v) = \sum_{\ell \in \bbZ} (r_\ell(u, u) - r_{\ell}(u, v)).
\end{equation*}
The main difficulty lies in establishing a bound on
$r_\ell(u, u) - r(u, v)$ which is both
$\frac{\beta}{2}$-Hölder in $(u, v)$ and summable in $\ell$.
Using the standard identity $\cos(x) = (\e^{\I x} + \e^{-\I x})/2$ and
the symmetry of $L_1$ we deduce the identity
\begin{equation*}
  2(r_{\ell}(u, u) - r_{\ell}(u, v)) = [U_{Z_0, Z_\ell}(u, -u) - U_{Z_0, Z_\ell}(u, -v)]
  + [U_{Z_0, Z_\ell}(u, u) - U_{Z_0, Z_\ell}(u, v)].
\end{equation*}
The two terms in the square brackets are treated very similarly so we
consider only the first one. Before diving into the tedious 
calculations we recall the following inequalities for $x, y \in \bbR$:
\begin{alignat}{2}
  \abs{\e^{-x} - \e^{-y}} &\leq \abs{x - y} &&\qquad \text{if
                                               $x, y \geq 0$}
                                               \label{eq:exp-lipschitz}
  \\
  \abs{x + y}^{\beta} &\leq \abs{x}^{\beta} + \abs{y}^{\beta} &&\qquad \text{for $\beta \in (0, 1]$}
                                                                 \label{eq:subadditivity}
  \\
  \abs{\abs{x}^{\beta} - \abs{y}^{\beta}} &\leq \abs{x - y}^{\beta} &&\qquad
                                                                       \text{for $\beta \in (0, 1]$}
                                                                       \label{eq:power-is-holder}
  \\
  \abs{\abs{x + y}^{\beta} - \abs{x}^{\beta} - \abs{y}^{\beta}}
                          &\leq \abs{x y}^{\beta/2} &&\qquad \text{for $\beta \in (0, 2)$.}
                                                       \label{eq:remainder-in-subadditive}
\end{alignat}
Define additionally the two quantities
\begin{equation*}
  \rho_i
  = \int_{\bbR} \abs{g_{\xi}(x) g_{\xi}(x + i)}^{\beta/2} \Di x
  \qquad \text{and} \qquad
  \mu_i
  = \int_{-m}^{\infty} \abs{g_{\xi}(x + i)}^\beta \Di x, \qquad i \in \bbZ.
\end{equation*}
We shall need the following lemma.
\begin{lemma}
  \label{lem:rho-mu-ineq}
  There exists a constant $C > 0$ such that for any $i \in \bbN$
  \begin{tlist}
  \item\label{it:lem:rho-mu-ineq:1}
    $\rho_i \leq C i^{-\alpha \beta/2}$.

  \item\label{it:lem:rho-mu-ineq:2} If $i > m$ then
    $\mu_i \leq C (i - m)^{1 - \alpha \beta}$.
  \end{tlist}
\end{lemma}

\begin{proof}
  \cref{it:lem:rho-mu-ineq:1} follows as in
  \cite[Lemma~4.1]{BassBerr}. For \cref{it:lem:rho-mu-ineq:2} note if
  $k > m$ then $x + k > 1$ for any $x > -m$, so according to
  assumption \cref{eq:kernel-ass}
  \begin{equation*}
    \mu_i \leq C \int_{-m}^{\infty} (x + k)^{-\alpha \beta} = C (k - m)^{1 - \alpha \beta},
  \end{equation*}
  where we used that $\alpha \beta > 2$.
\end{proof}

\noindent Using the expression for the characteristic function of a
symmetric $\beta$-stable random variable we decompose as follows
\begin{align*}
  \MoveEqLeft[0] U_{Z_0, Z_\ell}(u, - u) - U_{Z_0, Z_\ell}(u, -v)
  \\
  &= \exp\Bigl(-\norm[\Big]{\sum_{i = 1}^{m} u_i(g_{\xi}(i - \argmrk) - g_{\xi}(i + \ell - \argmrk) )}_{\beta}^{\beta} \Bigr)
    - \exp\Bigl(-2 \norm[\Big]{\sum_{i = 1}^{m} u_i g_{\xi}(i - \argmrk)}_{\beta}^{\beta} \Bigr)
  \\
  &- \biggl[\exp\Bigl(-\norm[\Big]{\sum_{i = 1}^{m} u_i g(i - \argmrk) - v_i g(i + \ell - \argmrk)}_{\beta}^{\beta} \Bigr)
  \\
  &-\exp\Bigl(-\norm[\Big]{\sum_{i = 1}^{m} u_i g_{\xi}(i - \argmrk)}_{\beta}^{\beta}
    - \norm[\Big]{\sum_{i = 1}^{m} v_i g_{\xi}(i + \ell - \argmrk)}_{\beta}^{\beta}\Bigr)
    \biggr]
  \\
  &= \exp\Bigl(2 \norm[\Big]{\sum_{i = 1}^{m} u_i g_{\xi}(i - \argmrk)}_{\beta}^{\beta} \Bigr)
  \\
  &\times \biggl[\exp\Bigl(-2 \norm[\Big]{\sum_{i = 1}^{m} u_i g_{\xi}(i - \argmrk)}_{\beta}^{\beta} \Bigr) - \exp\Bigl( - \norm[\Big]{\sum_{i = 1}^{m} u_i g_{\xi}(i - \argmrk)}_{\beta}^{\beta} - \norm[\Big]{\sum_{i = 1}^{m} v_i g_{\xi}(i - \argmrk)}_{\beta}^{\beta} \Bigr) \biggr]
  \\
  &\times \biggl[\exp\Bigl( - \norm[\Big]{\sum_{i = 1}^{m} u_i (g_{\xi}(i - \argmrk) - g_{\xi}(i + \ell - \argmrk))}_{\beta}^{\beta} \Bigr) - \exp\Bigl(-2\norm[\Big]{\sum_{i = 1}^{m} u_i g_{\xi}(i - \argmrk)}_{\beta}^{\beta}\Bigr) \biggr]
  \\
  &+ \exp\Bigl(-\norm[\Big]{\sum_{i = 1}^{m} u_i g_{\xi}(i - \argmrk)}_{\beta}^{\beta} - \norm[\Big]{\sum_{i = 1}^{m} v_i g_{\xi}(i - \argmrk)}_{\beta}^{\beta}\Bigr)
  \\
  &\times \biggl[ \exp\Bigl(-\norm[\Big]{\sum_{i = 1}^{m} u_i (g_{\xi}(i - \argmrk) - g_{\xi}(i + \ell - \argmrk)) }_{\beta}^{\beta} + 2 \norm[\Big]{\sum_{i = 1}^{m} u_i g_{\xi}(i - \argmrk) }_{\beta}^{\beta}\Bigr)
  \\
  &- \exp\Bigl(-\norm[\Big]{\sum_{i = 1}^{m} u_i g_{\xi}(i - \argmrk) - v_i g_{\xi}(i + \ell - \argmrk)}_{\beta}^{\beta} + \norm[\Big]{\sum_{i = 1}^{m} u_i g_{\xi}(i - \argmrk)}_{\beta}^{\beta} + \norm[\Big]{\sum_{i = 1}^{m} v_i g_{\xi}(i - \argmrk)}_{\beta}^{\beta}\Bigr)\biggr]
  \\
  &\eqqcolon r^1_{\ell}(u, v) + r_{\ell}^2(u, v).
\end{align*}
For the first term, $r_\ell^1$, we notice that the exponential term in
front is bounded in $u \in [0, T]^m$ (and of course in $\ell \in \bbZ$
as well), hence by \cref{eq:exp-lipschitz}
\begin{align*}
  r^1_{\ell}(u, v) &\leq C_T \abs[\bigg]{\norm[\Big]{\sum_{i = 1}^{m} u_i g_{\xi}(i - \argmrk) }_{\beta}^{\beta} - \norm[\Big]{\sum_{i = 1}^{m} v_i g_{\xi}(i - \argmrk) }_{\beta}^{\beta}}
  \\
                   &\quad \times \abs[\bigg]{\norm[\Big]{\sum_{i = 1}^{m} u_i (g_{\xi}(i - \argmrk) - g_{\xi}(i + \ell - \argmrk))}_{\beta}^{\beta} - 2\norm[\Big]{\sum_{i = 1}^{m} u_i g_{\xi}(i - \argmrk)}_{\beta}^{\beta}}.
\end{align*}
The first absolute value term will give the Hölder continuity of order
$\beta/2$ and the second will ensure summability in $\ell$. For the
first term we may bound as follows in the case $\beta \in (0, 1]$
using \cref{eq:power-is-holder,eq:subadditivity}
\begin{align*}
  \abs[\bigg]{\norm[\Big]{\sum_{i = 1}^{m} u_i g_{\xi}(i - \argmrk) }_{\beta}^{\beta} - \norm[\Big]{\sum_{i = 1}^{m} v_i g_{\xi}(i - \argmrk) }_{\beta}^{\beta}}
  &\leq \int_{\bbR} \Bigl(\sum_{i = 1}^{m} \abs{u_i - v_i} \abs{g_{\xi}(i - x)}\Bigr)^{\beta} \Di x
  \\
  &\leq \norm{u - v}^{\beta} \sum_{i = 1}^{m} \int_{\bbR} \abs{g_{\xi}(i - x)}^{\beta} \Di x
  \\
  &\leq C_T \norm{u - v}^{\beta / 2}.
\end{align*}
If instead $\beta > 1$, then the map is
$u \mapsto \norm{\sum_{i = 1}^{m} u_i g_{\xi}(i -
  \argmrk)}_{\beta}^{\beta}$ is continuously differentiable, hence by
the mean value theorem it is Hölder continuous of any order less than
or equal to~$1$, and since $\beta \in (0, 2)$ Hölder continuity of
order $\beta/2$ then holds. For the second absolute value term it
follows by \cref{eq:remainder-in-subadditive} and
\cref{eq:subadditivity}
\begin{align*}
  \MoveEqLeft \abs[\bigg]{\norm[\Big]{\sum_{i = 1}^{m} u_i (g_{\xi}(i - \argmrk)
  - g_{\xi}(i + \ell - \argmrk))}_{\beta}^{\beta}
  - 2 \norm[\Big]{\sum_{i = 1}^{m} u_i g_{\xi}(i - \argmrk)}_{\beta}^{\beta}}
  \\
  &= \biggl\lvert\norm[\Big]{\sum_{i = 1}^{m} u_i (g_{\xi}(i - \argmrk)
    - g_{\xi}(i + \ell - \argmrk))}_{\beta}^{\beta}
    - \norm[\Big]{\sum_{i = 1}^{m} u_i g_{\xi}(i - \argmrk)}_{\beta}^{\beta}
  \\
  &\quad - \norm[\Big]{- \sum_{i = 1}^{m} u_i g_{\xi}(i + \ell - \argmrk)}_{\beta}^{\beta}
    \biggr\rvert
  \\
  &\leq 2 \norm[\Big]{\Bigl(\sum_{i = 1}^{m} u_i g_{\xi}(i - \argmrk)\Bigr)
    \Bigl(\sum_{k = 1}^{m} u_k g_{\xi}(k + \ell - \argmrk)\Bigr)}_{\beta/2}^{\beta/2}
  \\
  &\leq 2 T^{\beta} \sum_{i, k = 1}^{m} \norm{g_{\xi}(i - \argmrk) g_{\xi}(k + \ell - \argmrk)}_{\beta/2}^{\beta/2}
  \\
  &= 2 T^{\beta} \sum_{i, k = 1}^{m} \rho_{\ell + k - i},
\end{align*}
which is summable in $\ell$ by \cref{lem:rho-mu-ineq} and the
assumption $\alpha \beta > 2$. We now turn our attention to the more
complicated second term $r^2_\ell(u, v)$. Utilizing
\cref{eq:exp-lipschitz} we have that
\begin{align*}
  r_\ell^2(u, v) &\leq \biggl\lvert \norm[\Big]{\sum_{i = 1}^{m} u_i (g_{\xi}(i - \argmrk)
                   - g_{\xi}(i + \ell - \argmrk))}_{\beta}^{\beta}
                   - 2 \norm[\Big]{\sum_{i = 1}^{m} u_i g_{\xi}(i - \argmrk) }_{\beta}^{\beta}
  \\
                 &\quad + \norm[\Big]{\sum_{i = 1}^{m} v_i g_{\xi}(i - \argmrk)}_{\beta}^{\beta}
                   + \norm[\Big]{\sum_{i = 1}^{m} u_i g_{\xi}(i  - \argmrk)}_{\beta}^{\beta}
  \\
                 &\quad - \norm[\Big]{\sum_{i = 1}^{m} u_i g_{\xi}(i - \argmrk) - v_i g_{\xi}(i + \ell - \argmrk)}_{\beta}^{\beta}
                   \biggr\rvert
  \\
                 &= \biggl\lvert \int_{-m}^{\infty} \biggl[ \abs[\Big]{\sum_{i = 1}^{m} u_i (g_{\xi}(x + i) - g_{\xi}(i + \ell + x))}^{\beta}
  \\
                 &\quad - \abs[\Big]{\sum_{i = 1}^{m} u_i g_{\xi}(i + x) - v_i g_{\xi}(i + \ell + x)}^{\beta} \biggr]
  \\
                 &\quad + \biggl[\abs[\Big]{\sum_{i = 1}^{m} v_i g_{\xi}(i + \ell + x)}^{\beta}
                   - \abs[\Big]{\sum_{i = 1}^{m} u_i g_{\xi}(i + \ell + x)}^{\beta}
                   \biggr]\Di x\biggr\rvert
  \\
                 &\leq \int_{-m}^{\infty} \biggl\lvert \abs[\Big]{\sum_{i = 1}^{m} u_i (g_{\xi}(x + i) - g_{\xi}(i + \ell + x))}^{\beta}
  \\
                 &\quad - \abs[\Big]{\sum_{i = 1}^{m} u_i g_{\xi}(i + x) - v_i g_{\xi}(i + \ell + x)}^{\beta} \biggr\rvert \Di x
  \\
                 &\quad + \int_{-m}^{\infty} \biggl\lvert \abs[\Big]{\sum_{i = 1}^{m} v_i g_{\xi}(i + \ell + x)}^{\beta}
                   - \abs[\Big]{\sum_{i = 1}^{m} u_i g_{\xi}(i + \ell + x)}^{\beta}
                   \biggr\rvert\Di x
  \\
                 &\eqqcolon r_\ell^{2, 1}(u, v) + r_\ell^{2, 2}(u, v).
\end{align*}
We deal first with the second term $r_\ell^{2, 2}$. First, if
$\beta \in (0, 1]$, then by \cref{eq:power-is-holder,eq:subadditivity}
\begin{equation*}
  r_\ell^{2, 2}(u, v) \leq \int_{-m}^{\infty} \abs[\Big]{\sum_{i = 1}^{m} (u_i - v_i) g_{\xi}(i + \ell + x)}^\beta \Di x
  \leq \norm{u - v}^{\beta} \sum_{i = 1}^{m} \mu_{i + \ell},
\end{equation*}
and by \cref{lem:rho-mu-ineq}\cref{it:lem:rho-mu-ineq:2} we obtain a
bound which is summable in $\ell > m$. If instead $\beta \in (1, 2)$
the map
\begin{equation*}
  h(u) = \int_{-m}^{\infty} \abs[\Big]{\sum_{i = 1}^{m} u_i g_{\xi}(i + \ell + x)}^{\beta} \Di x,
  \qquad u \in \bbR^m,
\end{equation*}
is continuously differentiable and the absolute value of the
derivative is bounded as follows for any $u \in [0, T]^m$ and
$\ell > m$:
\begin{align*}
  \abs[\Big]{\frac{\partial}{\partial u_k} h(u)}
  &\leq \int_{-m}^{\infty} \abs[\Big]{\sum_{i = 1}^{m} u_i g_{\xi}(i + \ell + x)}^{\beta - 1}
    \abs{g_{\xi}(k + \ell + x)} \Di x
  \\
  &\leq T^{\beta - 1} \sum_{i = 1}^{m} \int_{-m}^{\infty} \abs{g_{\xi}(i + \ell + x)}^{\beta - 1} \abs{g_{\xi}(k + \ell + x)} \Di x
  \\
  &\leq C T^{\beta - 1} m (\ell - m)^{1 - \alpha \beta},
\end{align*}
where we have argued as in
\cref{lem:rho-mu-ineq}\cref{it:lem:rho-mu-ineq:2} in the last
inequality. Hence, in the case $\beta \in (1, 2)$ we obtain by the
mean value theorem
\begin{equation*}
  r_\ell^{2, 2}(u, v) \leq \sup_{z \in [0, T]^m} \norm{\nabla h(z)} \norm{u - v}
  \leq C_T (\ell - m)^{1 - \alpha \beta} \norm{u - v}^{\beta},
\end{equation*}
and as $\alpha \beta > 2$ we have obtained a bound summable in $\ell$.

It remains to consider the term $r_\ell^{2, 1}$. Here it follows from
the inequality
$\abs{\abs{x}^{\beta} - \abs{y}^{\beta}} \leq \abs{x^2 - y^2}^{\beta /
  2}$ and the triangle inequality that the integrand is bounded by
\begin{align*}
  \MoveEqLeft \abs[\bigg]{\abs[\Big]{\sum_{i = 1}^{m} u_i(g_{\xi}(i + x) - g_{\xi}(i + \ell + x) )}^{\beta} - \abs[\Big]{\sum_{i = 1}^{m} u_ig_{\xi}(i + x) - v_i g_{\xi}(i + \ell + x)}^{\beta}}
  \\
  &\leq \biggl\lvert \sum_{i, k = 1}^{m} u_i u_k (g_{\xi}(i + x) - g_{\xi}(i + \ell + x))(g_{\xi}(k + x) - g_{\xi}(k + \ell + x))
  \\
  &\quad - (u_i g_{\xi}(i + x) - v_i g_{\xi}(i + \ell + x)) (u_k g_{\xi}(k + x) - v_k g_{\xi}(k + \ell + x))
    \biggr\rvert^{\beta/2}
  \\
  &= \biggl\lvert \sum_{i, k = 1}^{m} \Bigl[ (u_i u_k - v_i v_k) g_{\xi}(i + \ell + x) g_{\xi}(k + \ell + x)
  \\
  &\quad + u_i (v_k - u_k) g_{\xi}(i + x) g_{\xi}(k + \ell + x)
  \\
  &\quad + u_k (v_i - u_i) g_{\xi}(i + \ell + x) g_{\theta, \beta}(k + x)
    \Bigr]\biggr\rvert^{\beta/2}
  \\
  &\leq C_T \norm{u - v}^{\beta / 2} \sum_{i, k = 1}^{m} \Bigl[ \abs{g_{\xi}(i + \ell + x) g_{\xi}(k + \ell + x)}^{\beta / 2}  + \abs{g_{\xi}(i + x) g_{\xi}(k + \ell + x)}^{\beta / 2}\Bigr].
\end{align*}
Hence, we obtain with arguments as in
\cref{lem:rho-mu-ineq}\cref{it:lem:rho-mu-ineq:2} that
\begin{equation*}
  r_{\ell}^{2, 1}(u, v) \leq C_T \norm{u - v}^{\beta / 2} \Bigl((\ell - m)^{1 - \alpha \beta} + \sum_{i, k = 1}^{m} \rho_{\ell + k - i} \Bigr),
\end{equation*}
which is summable in $\ell$ as $\alpha \beta > 2$.

Lastly, we shall prove that $(G_u)_{u \in \bbR_+^m}$ has paths in
$\cL_w^1(\bbR_+^m)$ almost surely, such that
$\int_{\smash{\bbR_+^m}} G_u w(u) \Di u$ is well-defined. A sufficient
criteria for this is
$\int_{\smash{\bbR_+^m}} \var[G_u]^{\smash{1 / 2}} w(u) \Di u <
\infty$, since $G$ is centred. For this we need to study
$r_{\ell}(u, u)$ again. Recall that
\begin{equation*}
  r_\ell(u, u) = U_{Z_0, Z_\ell}(u, -u) + U_{Z_0, Z_\ell}(u, u).
\end{equation*}
As both terms are treated almost identically it suffices to consider
the first one. Here it follows from the inequality
$\abs{\e^x - 1} \leq \e^{\abs{x}} \abs{x}$, $x \in \bbR$, and
\cref{eq:remainder-in-subadditive}, that
\begin{align*}
  \MoveEqLeft \abs{U_{Z_0, Z_\ell}(u, -u)}
  \\
  &= \abs[\bigg]{\exp\Bigl(-\norm[\Big]{\sum_{i = 1}^{m} u_i (g_{\xi}(i - \argmrk) - g_{\xi}(i + \ell - \argmrk))}_{\beta}^{\beta} \Bigr)
    - \exp\Bigl(-2 \norm[\Big]{\sum_{i = 1}^{m} u_i g_{\xi}(i - \argmrk)}_{\beta}^{\beta}\Bigr)
    }
  \\
  &\leq \exp\Bigl(-2\norm[\Big]{\sum_{i = 1}^{m} u_i g_{\xi}(i - \argmrk)}_{\beta}^{\beta}\Bigr)
  \\
  &\quad \times \abs[\bigg]{\norm[\Big]{\sum_{i = 1}^{m} u_i (g_{\xi}(i - \argmrk) - g_{\xi}(i + \ell - x))}_{\beta}^{\beta} - 2 \norm[\Big]{\sum_{i = 1}^{m} u_i g_{\xi}(i - \argmrk)}_{\beta}^{\beta}}
  \\
  &\quad \times \exp\biggl(\abs[\bigg]{\norm[\Big]{\sum_{i = 1}^{m} u_i (g_{\xi}(i - \argmrk) - g_{\xi}(i + \ell - x))}_{\beta}^{\beta} - 2 \norm[\Big]{\sum_{i = 1}^{m} u_i g_{\xi}(i - \argmrk)}_{\beta}^{\beta}} \biggr)
  \\
  &\leq \exp\Bigl(-2 \norm[\Big]{\sum_{i = 1}^{m} u_i g_{\xi}(i - \argmrk)}_{\beta}^{\beta} + 2\norm[\Big]{\Bigl(\sum_{i = 1}^{m} u_i g_{\xi}(i - \argmrk) \Bigr) \Bigl(\sum_{i = 1}^{m} u_i g_{\xi}(i + \ell - \argmrk) \Bigr)}_{\beta/2}^{\beta/2}\Bigr)
  \\
  &\quad \times \norm{u}^\beta \sum_{i, k = 1}^{m} \rho_{\ell + k - i}
  \\
  &\leq \norm{u}^{\beta} \sum_{i, k = 1}^{m} \rho_{\ell + k - i},
\end{align*}
where we have used the Cauchy--Schwarz inequality in the last
line. Summing over $\ell$ yields an element in $\cL_w^1(\bbR_+^m)$ by
the assumption on the weight function $w$.

\subsubsection{Convergence of integral functionals} \label{sec4.1.2}

In \cref{sec:limit} we saw that the empirical characteristic functions
suitably scaled and centred converge to a Gaussian process in finite
dimensional sense. We wish to extend this convergence to integrals of
our processes. For this we need to extend \cite[Lemma~1]{LjunANot} to
a multivariate case. For $x \in \bbR$ let $\floor{x}$ denote the
largest integer $l$ such that $l \leq x$ and for a vector
$u = (u_1, \ldots, u_m) \in \bbR^m$ we let
$\floor{u} = (\floor{u_1}, \ldots, \floor{u_m})$.

\begin{lemma}
  \label{lem:int-conv}
  Let $(Y_u^n)_{u \in \bbR_+^m}$ and $(Y_u)_{u \in \bbR_+^m}$ be
  continuous random fields with $Y^n \fidiTo Y$. Assume that
  $\int_{\bbR_+^m} \Expt{\abs{Y_u^n}} \Di u < \infty$ and
  $\int_{\bbR_+^m} \Expt{\abs{Y_u}} \Di u < \infty$, and set for
  $k, \ell, n \in \bbN$
  \begin{equation*}
    X_{n, k, \ell} = \int_{[0, \ell]^m} Y^n_{\floor{u k}/k} \Di u
    \qtq{and}
    X_{n, \ell} = \int_{[0, \ell]^m} Y^n_u \Di u.
  \end{equation*}
  Suppose that
  \begin{equation*}
    \lim_{\ell \to \infty} \limsup_{n \to \infty} \int_{\bbR_+^{m - 1 - i}} \int_{\ell}^{\infty} \int_{\bbR_+^i} \Expt{\abs{Y_u^n}} \Di u = 0, \qquad
    \lim_{k \to \infty} \limsup_{n \to \infty} \bbP(\abs{X_{n, k, \ell} - X_{n, \ell}} > \varepsilon) = 0
  \end{equation*}
  where the first convergence holds for all
  $i \in \Set{0, \ldots, m - 1}$ and the latter for all
  $\varepsilon, \ell > 0$. Then convergence in distribution holds:
  \begin{equation*}
    \int_{\bbR^m_+} Y_u^n \Di u \distTo \int_{\bbR^m_+} Y_u \Di u \qquad \text{as $n \to \infty$.}
  \end{equation*}
\end{lemma}

\begin{proof}
  Observe for each $\ell > 0$ the decomposition
  \begin{equation*}
    \int_{\bbR^m_+} Y_u^n \Di u = X_{n, k, \ell} + (X_{n, \ell} - X_{n, k, \ell})
    + \sum_{i = 0}^{m - 1} \int_{\bbR_+^{m - 1 - i}} \int_{\ell}^{\infty} \int_{[0, \ell]^i} Y_u^n \Di u.
  \end{equation*}
  Conclude now as in \cite[Lemma~1]{LjunANot}.
\end{proof}

\subsubsection{Convergence of the estimator}
\label{sec4.1.3}

First, $\xi_n \asTo \xi_0$ follows by standard arguments which in
particular uses Assumption~(A), see, e.g. \cite{LjunANot}, where one
uses
\begin{equation}
  \label{eq:norm-conv}
  \norm{\varphi_n - \varphi_{\xi_0}}_w \asTo 0 \qquad \text{as $n \to \infty$,}
\end{equation}
which is a consequence of Lebesgue's dominated convergence
theorem. Indeed, denote by $\lambda$ the Lebesgue measure on
$\bbR^m$. For \cref{eq:norm-conv} it is by dominated convergence
enough to prove that there exists a $\bbP$-null set $N$ such that for
all $\omega \in \Omega \setminus N$
\begin{equation*}
  \varphi_n(u, \omega) \to \varphi_{\xi_0}(u) \qquad \text{for $\lambda$-almost all $u \in \bbR_+^m$.}
\end{equation*}
To see this set
$A = \Set{(\omega, u) \mid \varphi_n(u, \omega) \not\to
  \varphi_{\xi_0}(u)}$ and note that by Tonelli's~theorem:
\begin{align*}
  \int_{\Omega} \lambda(\Set{u \mid \varphi_n(u, \omega) \not\to \varphi_{\xi_0}(u) }) \isp \bbP(\di \omega)
  &= \int_{\Omega} \int_{\bbR^m} \1_A(\omega, u) \isp \lambda(\di u) \isp \bbP(\di \omega)
  \\
  &= \int_{\bbR^m} \bbP(\Set{\omega \mid \varphi_n(u, \omega) \not\to \varphi_{\xi_0}(u)}) \isp \lambda(\di u)
    = 0,
\end{align*}
where the last equality follows from Birkhoff's ergodic theorem which
states that for each $u \in \bbR^m$, then
\begin{equation*}
  \varphi_n(u) \to \varphi_{\xi_0}(u) \qquad \text{$\bbP$-almost surely as $n \to \infty$.}
\end{equation*}

To derive the central limit theorem for the estimator, we consider
instead the requirement
\begin{equation*}
  \nabla_{\xi} F(\varphi, \xi) = 0 \qquad \varphi \in \cL_w^2(\bbR_+^m), \quad \xi \in \Xi
\end{equation*}
which is satisfied at $(\varphi_{\xi_0}, \xi_0)$. The problem may now
be viewed from a implicit functional point of view. To this end we
recall the implicit function theorem on general
Banach~spaces. Consider a Fréchet differentiable map
$g : U_1 \times U_2 \to B_3$ where $U_1$ and $U_2$ are open subsets of
the Banach spaces $B_1$ and $B_2$, respectively, and $B_3$ is an
additional Banach space. Let~$D^i_{h_i} g(p_1, p_2)$,
$i \in \Set{1, 2}$, denote the partial derivatives at the point
$(p_1, p_2) \in U_1 \times U_2$ in the direction $h_i \in B_i$. If
$(p_1^0, p_2^0) \in U_1 \times U_2$ is a point such that
$g(p_1^0, p_2^0) = 0$ and the map
$h \mapsto D^2_h g(p_1^0, p_2^0) : B_2 \to B_3$ is a continuous and
invertible function, then there exists open subsets
$V_1 \subseteq U_1$ and $V_2 \subseteq U_2$ such that
$(p_1^0, p_2^0) \in V_1 \times V_2$ and a Fréchet differentiable and
bijective (implicit) function $\Phi : V_1 \to V_2$ such that
\begin{equation*}
  g(p_1, p_2) = 0 \iffx \Phi(p_1) = p_2.
\end{equation*}
In addition, the derivative is given by
\begin{equation}
  \label{eq:derivative-implicit}
  D_h \Phi(p) = - \bigl(D^2 g(p, \Phi(p)) \bigr)^{-1} \bigl(D_h^1 g(p, \Phi(p)) \bigr), \qquad h \in B_1, \quad p \in V_1.
\end{equation}
As might be apparent we shall consider the specific setup of
$g = \nabla_{\xi} F$, $B_1 = U_1 = \cL^2_w(\bbR_+^m)$,
$U_2 = \Xi \subseteq B_2 = \bbR^{d + 1}$. We note that
Assumption~(B)\cref{it:ass:B:2} ensures the existence and continuity
of the first and second order derivatives of $F$. Moreover,
Assumption~(A)\cref{it:ass:A:4} yields the invertibility of the
Hessian $\nabla^2_\xi F(\varphi_{\xi_0}, \xi_0)$.

In this case
\begin{equation*}
  \Phi(\varphi_n) = \xi_n \qtq{and} \Phi(\varphi_{\xi_0}) = \xi_0.
\end{equation*}
Hence, by Fréchet differentiability we find that
\begin{align*}
  \sqrt{n} (\xi_n - \xi_0) &= \sqrt{n} (\Phi(\varphi_{\xi_0} + (\varphi_n - \varphi_{\xi_0}))
                             - \Phi(\varphi_{\xi_0}))
  \\
                           &=  D_{\sqrt{n} (\varphi_n - \varphi_{\xi_0})} \Phi(\varphi_{\xi_0})
                             + \sqrt{n} \norm{\varphi_n - \varphi_{\xi_0}}_{w, 2} R(\varphi_n - \varphi_{\xi_0}),
\end{align*}
where the remainder term satisfies that
$R(\varphi_n - \varphi_{\xi_0}) \asTo 0$ as
$\norm{\varphi_n - \varphi_{\xi_0}}_{w, 2} \asTo 0$. Recalling the
derivative at \cref{eq:derivative-implicit} and the representation
$F(\varphi, \xi) = \iprod{\varphi - \varphi_\xi, \varphi -
  \varphi_\xi}_w$, it suffices to prove that
\begin{equation*}
  \begin{aligned}
    \sqrt{n} \norm{\varphi_n - \varphi_{\xi_0}}_{w, 2} &\distTo \norm{G}_{w, 2}
    \\
    (\iprod{\partial_{\xi}^i \varphi_{\xi_0}, \sqrt{n}(\varphi_n - \varphi_{\xi_0})}_w)_{i = 1, \ldots, d + 1} &\distTo (\iprod{\partial_{\xi}^i \varphi_{\xi_0}, G}_w)_{i = 1 \ldots, d + 1}.
  \end{aligned}
\end{equation*}
We focus on the last convergence since they are shown similarly. For
this we wish to use (a vector version of) \cref{lem:int-conv} which
requires the finite dimensional convergence of the vector-valued
process
\begin{equation*}
  Z_u^n = \bigl(\partial_{\xi}^i \varphi_{\xi_0}(u) w(u) \sqrt{n} (\varphi_n(u) - \varphi_{\xi_0}(u))\bigr)_{i = 1, \ldots, d + 1}.
\end{equation*}
But since it is the same underlying process,
$(\sqrt{n} (\varphi_n(u) - \varphi_{\xi_0}(u)))_{u \in \bbR_+^m}$,
this simply follows from the continuous mapping theorem in conjunction
with the finite dimensional convergence observed at
\cref{eq:fidi-conv-char-functions}. A small generalization of
\cref{lem:int-conv} shows that is sufficient to provide suitable
moment estimates for each individual coordinate, that is, estimates
for
\begin{equation*}
  Y_u^n \coloneqq \partial_{\xi}^i \varphi_{\xi_0}(u) w(u) \sqrt{n} (\varphi_n(u) - \varphi_{\xi_0}(u))
  \eqqcolon h(u) G_u^n, \qquad u \in \bbR_+^m, \quad n \in \bbN,
\end{equation*}
where $i \in \Set{1, \ldots, d + 1}$ is fixed and $h$ and $G^n$ are
defined respectively as
\begin{equation*}
  h(u) = \partial_{\xi}^i \varphi_{\xi_0}(u) w(u) \qtq{and}
  G_u^n = \sqrt{n} (\varphi_n(u) - \varphi_{\xi_0}(u)).
\end{equation*}
Note that $h$ is continuous by Assumption~(A)\cref{it:ass:A:4} and
since the weight function is continuous.

Using arguments as in \cite[Section~4.2]{LjunANot} and the variance
estimates from \cref{sec:limit} we deduce that
\begin{equation*}
  \Expt[\big]{\abs{Y_u^n}^2}
  \leq (\partial_{\xi}^i \varphi_{\xi_0}(u) w(u))^2 \sum_{\ell \in \bbZ} \abs{r_\ell(u, u)}
  \leq C \norm{u}^{\beta} (\partial_{\xi}^i \varphi_{\xi_0}(u) w(u))^2.
\end{equation*}
Taking the square root we obtain a bound in $\cL^1(\bbR_+^m)$ of
$\Expt{\abs{Y_u^n}}$ by the Cauchy--Schwarz inequality used together with
Assumption~(B)\cref{it:ass:B:2} and that $u \mapsto \norm{u}$ is an
element of $\cL_w^2(\bbR_+^m)$. Hence the first condition of
\cref{lem:int-conv} is satisfied. The second condition is slightly
more involved, but let a $\ell > 0$ be given and consider any
$u, v \in [0, \ell]^m$. Then
\begin{align*}
  \Expt[\big]{\abs{Y_u^n - Y_v^n}^2}^{1/2} &\leq \abs{h(u) - h(v)} \var[G_u^n]^{1/2}
                                             + \abs{h(v)} \cov(G_u^n, G_v^n)^{1/2}
  \\
                                           &\leq C_\ell(\abs{h(u) - h(v)} + \norm{u - v} ),
\end{align*}
which by Markov's inequality yields the second condition of
\cref{lem:int-conv}.

\subsection{Proof of Remark~\ref{rem:ass:A}\ref{it:rem:ass:A:4}}
\label{sec:remark-proof}

Let $\Set{g_\xi \given \xi \in \Xi}$ be a family of measurable
functions such that conditions \cref{it:ass:A:1,it:ass:A:3} from
Assumption~(A) holds. Then condition~(A)\cref{it:ass:A:4} holds if and
only if
$u \mapsto \partial_\xi^2 \varphi_\xi(u), \ldots, \partial_{\xi}^{d +
  1} \varphi_{\xi}(u)$ are linearly independent. The only if part is
trivial so we consider solely the if statement. Suppose therefore that
there exists constants $a_1, \ldots, a_{d + 1} \in \bbR$ such that
\begin{equation}
  \label{eq:linear-independence}
  a_1 \partial_\xi^1 \varphi_\xi(u) + \cdots + a_{d + 1} \partial_\xi^{d + 1} \varphi_\xi(u) = 0 \qquad \text{for all $u \in \bbR_+^m$.}
\end{equation}
We note first that for any $i \in \Set{1, \ldots, d + 1}$ and
$u_1 > 0$
\begin{align}
  \partial^i_\xi \varphi_\xi(u_1, 0, \ldots, 0)
  &= \varphi_\xi(u_1, 0, \ldots, 0) \partial_\xi^i \bigl(u_1^\beta \norm{g_\xi}_{\beta}^{\beta}\bigr)
    \nonumber
  \\
  &= \varphi_\xi(u_1, 0, \ldots, 0) \times
    \begin{cases}
      u_1^\beta \partial_{\xi}^i \norm{g_\xi}_\beta^\beta, & \text{if
        $i \neq 1$,}
      \\
      u_1^\beta \log(u_1) \norm{g_\xi}_\beta^{\beta} + u_1^\beta \partial_\xi^1 \norm{g_\xi}_\beta^\beta, & \text{if $i = 1$.}
    \end{cases}
                                                                                                            \label{eq:phi-derivative}
\end{align}
Since $\varphi_\xi(u) \neq 0$ for all $u \in \bbR_+^m$ it follows from
\cref{eq:phi-derivative} that
\begin{equation*}
  \frac{\partial_\xi^i \varphi_\xi(u_1, 0, \ldots, 0)}{\varphi_\xi(u_1, 0, \ldots, 0) u_1^\beta \log(u_1)} \xrightarrow[u_1 \to \infty]{}
  \begin{cases}
    0, & \text{if $i \neq 1$,}
    \\
    \norm{g_\xi}_\beta^\beta, & \text{if $i = 1$.}
  \end{cases}
\end{equation*}
This proves that $a_1 = 0$ using condition~(A)\cref{it:ass:A:1}. Then
\cref{eq:linear-independence} reduces to
\begin{equation*}
  a_2 \partial_\xi^2 \varphi_\xi(u) + \cdots + a_{d + 1} \partial_\xi^{d + 1} \varphi_\xi(u) = 0 \qquad \text{for all $u \in \bbR_+^m$,}
\end{equation*}
and $a_2 = \cdots = a_{d + 1} = 0$ follows by the assumption of linear
independence of the subset.

\subsection{Proof of statements in Example~\ref{ex:periodic-example}}
\label{sec:example-proofs}

Consider the kernel\footnote{Similar considerations can be done for
  the Ornstein--Uhlenbeck kernel, albeit easier and more explicit.}
$g_\theta(u) = \exp(- \theta_1 u - \theta_2 f(u))\1_{(0,\infty)}(u)$
for $\theta = (\theta_1, \theta_2) \in (0, \infty)^2$ and where $f$ is
a bounded measurable $1$-periodic function which does not vanish
except on a Lebesgue null set. Assume moreover that $f$ is either
non-positive and or non-negative. It is straightforward to see that in
this case the characteristic function of $X_1$ does not determine the
parameter $\theta$ uniquely. Consider instead the joint characteristic
function $\varphi_{\beta, \theta}(u_1, u_2)$ of $(X_1, X_2)$ for the
moving average $X$ with kernel $g_\theta$, which is given by:
\begin{equation*}
  \varphi_{\beta, \theta}(u_1, u_2)
  = \exp\bigl(-\norm{u_1 g_\theta + u_2 g_\theta(\cdot + 1)}_\beta^\beta \bigr),
  \qquad u_1, u_2 \geq 0.
\end{equation*}
If $\varphi_{\beta, \theta} = \varphi_{\beta, \tilde{\theta}}$ for
$\theta, \tilde{\theta} \in (0, \infty)^2$, then the $\beta$-norms
must be equal. Recalling the generalized binomial theorem
\begin{equation*}
  (x + y)^\beta = \sum_{k = 0}^{\infty} \binom{\beta}{k} x^{\beta - k} y^k \qquad x > y \geq 0
\end{equation*}
we may calculate these norms explicitly for $u_1 > u_2 \geq 0$:
\begin{align*}
  \MoveEqLeft[1] \norm{u_1 g_\theta + u_2 g_\theta(\cdot + 1)}_\beta^\beta
  \\
  &= u_2^\beta \int_{0}^{1} \exp(-\theta_1 x - \theta_2 f(x)) \Di x
  \\
  &\quad + \int_{0}^{\infty} \sum_{k = 0}^{\infty} \binom{\beta}{k} u_1^{\beta - k} u_2^k \exp(-(\beta - k) (\theta_1 x + \theta_ 2 f(x)) - k (\theta_1 (x + 1) + \theta_2 f(x + 1))) \Di x
  \\
  &= u_2^\beta \int_{0}^{1} \exp(-\theta_1 x - \theta_2 f(x)) \Di x
  \\
  &\quad + \int_{0}^{\infty} \sum_{k = 0}^{\infty} \binom{\beta}{k} u_1^{\beta - k} u_2^k \exp(-\beta (\theta_1 x + \theta_ 2 f(x)) - k \theta_1) \Di x
  \\
  &= u_2^\beta \int_{0}^{1} \exp(-\theta_1 x - \theta_2 f(x)) \Di x
    + (u_1 + u_2 \exp(-\theta_1))^\beta \int_{0}^{\infty} \exp(-\beta(\theta_1 x + \theta_2 f(x))) \Di x
\end{align*}
where the last equality follows from the generalized binomial theorem
since $u_1 > u_2 \geq u_2 \exp(-\theta_1)$. Hence if
$\varphi_{\beta, \theta} = \varphi_{\beta, \tilde{\theta}}$ then for
all $u_1 > u_2 \geq 0$
\begin{equation*}
  1 = \frac{u_2^\beta \int_{0}^{1} \exp(-\theta_1 x - \theta_2 f(x)) \Di x
    + (u_1 + u_2 \exp(-\theta_1))^\beta \int_{0}^{\infty} \exp(-\beta(\theta_1 x + \theta_2 f(x))) \Di x}{u_2^\beta \int_{0}^{1} \exp(-\tilde{\theta}_1 x - \tilde{\theta}_2 f(x)) \Di x
    + (u_1 + u_2 \exp(-\tilde{\theta}_1))^\beta \int_{0}^{\infty} \exp(-\beta(\tilde{\theta}_1 x + \tilde{\theta}_2 f(x))) \Di x}.
\end{equation*}
Inserting $u_1 = 1 > 0 = u_2$ yields the identity:
\begin{equation*}
  K \coloneqq  \int_{0}^{\infty} \exp(-\beta(\theta_1 x + \theta_2 f(x))) \Di x
  = \int_{0}^{\infty} \exp(-\beta(\tilde{\theta}_1 x + \tilde{\theta}_2 f(x))) \Di x
\end{equation*}
hence it suffices to prove that $\theta_1 =
\tilde{\theta}_1$. Moreover, inserting the above identity in
$\varphi_{\beta, \theta} = \varphi_{\beta, \tilde{\theta}}$ and
differentiating with respect to $u_1$ gives that for all $u_1 > u_2$:
\begin{equation*}
  (u_1 + u_2 \exp(-\theta_1))^{\beta - 1} K = (u_1 + u_2 \exp(-\tilde{\theta}_1))^{\beta - 1} K,
\end{equation*}
which proves that $\theta_1 = \tilde{\theta}_1$ if $\beta \neq 1$.

Let us additionally show that $u \mapsto \partial_{\xi}^2 \varphi_\xi$
and $u \mapsto \partial_{\xi}^3 \varphi_\xi$ are linearly independent
if the $1$-periodic function is negative and bounded and
$\beta \neq 1$. Indeed, by \cref{rem:ass:A}\cref{it:rem:ass:A:4} this
is equivalent to Assumption~(A)\cref{it:ass:A:4}. Due to their
exponential form these derivatives are linearly independent \emph{if}
the following functions (note that we only have an explicit formula
when $u_1 > u_2 \geq 0$) are linearly independent in
$u_1 > u_2 \geq 0$:
\begin{align*}
  \partial_{\xi}^2 \norm{u_1 g_\theta + u_2 g_{\theta}(\cdot + 1)}_\beta^\beta
  &= - K_{\theta, 1} u_2^\beta  - K_{\theta, 2} (u_1 + u_2)^{\beta - 1} u_2
    - K_{\theta, 3} (u_1 + u_2 \exp(-\theta_1))^\beta
  \\
  \partial_{\xi}^3 \norm{u_1 g_\theta + u_2 g_{\theta}(\cdot + 1)}_\beta^\beta
  &= K_{\theta, 4} u_2^\beta + K_{\theta, 5} (u_1 + u_2 \exp(-\theta_1))^\beta,
\end{align*}
where the constant $K_{\theta, 1}, \ldots, K_{\theta, 5}$ are strictly
positive, indeed the only constants which are not in general positive
are:
\begin{align*}
  K_{\theta, 5} &= - \int_{0}^{\infty} \beta \theta_2 f(x) \exp(-\beta(\theta_1 x + \theta_2 f(x))) \Di x
  \\
  K_{\theta, 4} &= - \int_{0}^{1}  f(x) \exp(-\theta_1 x - \theta_2 f(x)) \Di x
\end{align*}
but they are by our assumption $f < 0$. The main observation needed is
that these functions are of different order in $u_1$ when $u_2 \neq 0$
and that their constants are of opposite sign. Indeed, for
$a, b \in \bbR$ we have that
\begin{align*}
  0 &= \bigl(a \partial_{\xi}^2 \norm{u_1 g_\theta + u_2 g_{\theta}(\cdot + 1)}_\beta^\beta + b \partial_{\xi}^3 \norm{u_1 g_\theta + u_2 g_{\theta}(\cdot + 1)}_\beta^\beta \bigr) \bigm/ u_1^\beta
  \\
    &\xrightarrow[u_1 \to \infty]{}
      -a K_{\theta, 3} + b K_{\theta, 5}.
\end{align*}
The constants $a K_{\theta, 3}$ and $b K_{\theta, 5}$ must then be
same and we have the following major simplification:
\begin{align*}
  0 &= a \partial_{\xi}^2 \norm{u_1 g_\theta + u_2 g_{\theta}(\cdot + 1)}_\beta^\beta
      + b \partial_{\xi}^3 \norm{u_1 g_\theta + u_2 g_{\theta}(\cdot + 1)}_\beta^\beta
  \\
    &= - (a K_{\theta, 1} - b K_{\theta, 4}) u_2^\beta  - a K_{\theta, 2} (u_1 + u_2)^{\beta - 1} u_2.
\end{align*}
If $\beta > 1$ then this is clearly unbounded in $u_1$, hence
$a = 0$, and therefore $b = 0$ as well since $K_{\theta, 4} > 0$. If
$\beta < 1$ then differentiating with respect to $u_1$ yields the
simple equation:
\begin{equation*}
  0 = a K_{\theta, 2} (u_1 + u_2)^{\beta - 2} u_ 2 \qquad \text{for all $u_1 > u_2 \geq 0$,}
\end{equation*}
which yields $a = 0$ and therefore $b = 0$ since again
$K_{\theta, 4} > 0$.

\subsection{Proof of statements in Example~\ref{ex:modulated-OU}}
\label{sec:proof-modulated-OU}

Recall the moving average kernel from \cref{eq:modulated-OU}. First,
we show that the one-dimensional characteristic function is not enough
to identify $\theta = (\theta_1, \theta_2)$. Indeed, we see that for
two parameters $\theta$, $\tilde{\theta} \in (0, \infty)^2$ equality
of the one-dimensional characteristic functions gives
\begin{equation}
  \label{eq:scale-id}
  \begin{aligned}
    \frac{\theta_1^\beta \Gamma(\beta + 1)}{(\beta \theta_2)^{\beta + 1}}
    &= \int_{0}^{\infty} (\theta_1 s \exp(-\theta_2 s) )^\beta \Di s
    \\
    &= \int_{0}^{\infty} (\tilde{\theta}_1 s \exp(-\tilde{\theta}_2 s) )^\beta \Di s
      =   \frac{\tilde{\theta}_1^\beta \Gamma(\beta + 1)}{(\beta \tilde{\theta}_2)^{\beta + 1}}.
  \end{aligned}
\end{equation}
We claim that the two-dimensional characteristic function is enough to
identify $\theta$. For~this we recall the covariation between $X_1$
and $X_0$, cf. \cite[Section~2.7]{SamoStab}, which is uniquely
determined by the distribution of $(X_1, X_0)$ and hence by its joint
characteristic function. If $\theta$ denotes the underlying parameter
for the moving average $X$ and $\beta>1$, then the covariation is,
cf.~\cite[Proposition~3.5.2]{SamoStab},
\begin{equation}
  \label{eq:cov-id}
  \begin{aligned}
    [X_1, X_0]_\beta &= \int_{\bbR} g_\theta(s + 1) g_\theta(s)^{\beta - 1} \Di s
                       = \theta_1^\beta \int_{0}^{\infty} (s + 1) \e^{-\theta_2 (s + 1)} s^{\beta - 1} \e^{- (\beta - 1) \theta_2 s} \Di s                
    \\
                     &= \theta_1^\beta \e^{-\theta_2} \Bigl[ \int_{0}^{\infty} s^\beta \e^{-\beta \theta_2 s} \Di s
                       + \int_{0}^{\infty} s^{\beta - 1} \e^{-\beta \theta_2 s} \Di s
                       \Bigr]                   
    \\
                     &= \theta_1^\beta \e^{-\theta_2} \Bigl[\frac{\Gamma(\beta + 1)}{(\beta \theta_2)^{\beta + 1}}
                       + \frac{\Gamma(\beta)}{(\beta \theta_2)^\beta}
                       \Bigr]
    \\
                     &= \frac{\theta_1^\beta \Gamma(\beta + 1)}{(\beta \theta_2)^{\beta + 1}} \e^{-\theta_2} (1 + \theta_2),
  \end{aligned}
\end{equation}
where we used the defining property:
$\beta \Gamma(\beta) = \Gamma(\beta + 1)$. Hence if $\theta$ and
$\tilde{\theta}$ leads to the same distribution of $(X_1, X_0)$, then
combining the identities \cref{eq:scale-id,eq:cov-id} yields
\begin{equation*}
  (1 + \theta_2) \e^{-\theta_2} = (1 + \tilde{\theta}_2) \e^{-\tilde{\theta}_2}.
\end{equation*}
It is straightforward to check that the function
$x \mapsto (1 + x) \e^{-x}$ is strictly decreasing on $(0, \infty)$,
and therefore injective, which proves that
$\theta_2 = \tilde{\theta}_2$ and therefore
$\theta_1 = \tilde{\theta}_1$ as well, cf.~\cref{eq:scale-id}.

Let us now check the condition~(A)\cref{it:ass:A:4}. According to
\cref{rem:ass:A}\cref{it:rem:ass:A:4} it suffices to prove linear
independence of the functions $\partial_{\xi}^2 \varphi_\xi$ and
$ \partial_{\xi}^3 \varphi_\xi$. We obtain the identities
\begin{align*}
  \partial_{\xi}^2 \norm{u_1 g_\theta + u_2 g_{\theta}(\argmrk + 1)}_\beta^\beta
  &= \beta \theta_1^{\beta-1}
    \int_{\bbR} \Bigl(u_1 x \exp(-\theta_2 x) \1_{(0, \infty)}(x)
  \\
  &\quad+ u_2 (x + 1) \exp(-\theta_2 (x + 1)) \1_{(0, \infty)}(x + 1) \Bigr)^{\beta} \Di x,
  \\
  \partial_{\xi}^3 \norm{u_1 g_\theta + u_2 g_{\theta}(\argmrk + 1)}_\beta^\beta
  &= -\beta \theta_1^{\beta}
    \int_{\bbR} \Bigl(u_1 x \exp(-\theta_2 x) \1_{(0, \infty)}(x)
  \\
  &\quad + u_2 (x + 1) \exp(-\theta_2 (x + 1)) \1_{(0, \infty)}(x+1) \Bigr)^{\beta - 1}
  \\
  &\quad \times \Bigl( u_1 x^2 \exp(-\theta_2 x) \1_{(0, \infty)}(x)
  \\
  &\quad + u_2 (x + 1)^2\exp(-\theta_2 (x + 1)) \1_{(0, \infty)}(x + 1) \Bigr) \Di x.
\end{align*}
Notice that it suffices to show linear independence of the functions
\begin{equation*}
  f_1(u_1, u_2) \coloneqq \frac{\partial_{\xi}^2 \norm{u_1 g_\theta + u_2 g_{\theta}(\cdot + 1)}_\beta^\beta}{\beta \theta_1^{\beta-1}},
  \qquad
  f_2(u_1, u_2) \coloneqq \frac{\partial_{\xi}^3 \norm{u_1 g_\theta + u_2 g_{\theta}(\cdot + 1)}_\beta^\beta}{\beta \theta_1^{\beta}}.
\end{equation*}
Assume that there exist a constant $r$ such that
$f_1(u_1, u_2) + r f_2(u_1, u_2) = 0$ for any $u_1 > u_2 \geq
0$. Next, setting $u_2 = 0$, we obtain the identities
\begin{align*} 
  f_1(u_1, 0) &= u_1^{\beta} \int_0^{\infty} x^{\beta} \exp(-\beta \theta_2 x) \Di x
                = u_1^{\beta} (\beta \theta_2)^{-\beta -1} \Gamma(\beta + 1),
  \\
  f_2(u_1, 0) &= -u_1^{\beta} \int_0^{\infty} x^{\beta + 1} \exp(-\beta \theta_2 x) \Di x
                = -u_1^{\beta} (\beta \theta_2)^{-\beta -2} \Gamma(\beta + 2). 
\end{align*} 
Hence, it must hold that 
\begin{equation*}
  r = \frac{\beta \theta_2}{\beta + 1}.
\end{equation*}
In the next step we will show that
$f_1(u_1, \exp(\theta_2)) + r f_2(u_1, \exp(\theta_2)) \to \infty$ as
$u_1 \to \infty$, which leads to the desired contradiction. Recall
that $f_1(u_1, 0) + r f_2(u_1, 0) = 0$ and hence we may instead
consider
$f_1(u_1, \exp(\theta_2)) - f_1(u_1, 0) + r (f_2(u_1, \exp(\theta_2))
- f_2(u_1, 0))$. Applying~the mean value theorem we conclude that
\begin{equation*}
  f_1(u_1, \exp(\theta_2)) - f_1(u_1, 0)
  = \int_0^{1} x^{\beta} \exp(-\beta \theta_2 (x - 1)) \Di x + u_1^{\beta - 1} q_1
  + o(u_1^{\beta - 1}), 
\end{equation*}
where
\begin{equation*}
  q_1 = \beta \int_0^{\infty} x^{\beta - 1} (x + 1) \exp(-\beta \theta_2 x) \Di x.
\end{equation*}
Similarly, we deduce that
\begin{equation*}
  f_2(u_1,\exp(\theta_2)) - f_2(u_1, 0)
  = -\int_0^{1} x^{\beta + 1} \exp(-\beta \theta_2 (x - 1)) \Di x
  + u_1^{\beta - 1} q_2 + o(u_1^{\beta - 1})
\end{equation*}
with
\begin{equation*}
  q_2 = -\int_0^{\infty} \bigl(\beta x^{\beta + 1}
  + (\beta + 1) x^{\beta}
  + x^{\beta - 1} \bigr) \exp(-\beta \theta_2 x) \Di x.
\end{equation*}
Since $u_1^{\beta - 1} \to \infty$ as $u_1 \to \infty$ because
$\beta > 1$, we only need to prove that $q_1 + r q_2 \neq 0$. We have
that
\begin{align*}
  q_1 &= \beta \bigl( (\beta \theta_2)^{-\beta - 1} \Gamma(\beta+1) + (\beta \theta_2)^{-\beta} \Gamma(\beta) \bigr),
  \\
  q_2 &= - \bigl(\beta (\beta \theta_2)^{-\beta - 2} \Gamma(\beta + 2) + (\beta + 1) (\beta \theta_2)^{-\beta - 1} \Gamma(\beta + 1) + (\beta \theta_2)^{-\beta} \Gamma(\beta) \bigr).
\end{align*}
A straightforward calculation shows that
\begin{equation*}
  q_1 + r q_2
  = - r (\beta \theta_2)^{-\beta} \Gamma(\beta) < 0.
\end{equation*}
Consequently, we have a contradiction and the functions $f_1$ and
$f_2$ are linearly independent.

\subsection{Proof of statements in Example~\ref{ex:CARMA}}
\label{sec:proof-CARMA}

\noindent We consider a CARMA($2, 1$) model of the form
\begin{equation*}
  X_t = \int_{-\infty}^{t} b^{\top} \exp(A(t - s)) e \Di L_s, \qquad t \in \bbR,
\end{equation*}
where $b = (b_0, 1)^{\top}$, $e = (0, 1)^{\top}$, $L$ is a symmetric
$\beta$-stable Lévy process with $\beta \in (1,2)$, and
\begin{equation*}
  A = 
  \begin{pmatrix}
    0 & 1
    \\
    -\lambda^2 & 2\lambda
  \end{pmatrix}
\end{equation*}
with $\lambda<0 $. We further assume that $\theta=b_0+\lambda>0$. 
Recall the definition of the incomplete gamma function:
\begin{equation*}
  \Gamma(\beta;x) = \int_x^{\infty} y^{\beta-1} \exp(-y) \Di y, \qquad \beta, x > 0.
\end{equation*}
The following identity is due to partial integration:
$\Gamma(\beta + 1; x) = \beta \Gamma(\beta; x) + x^{\beta} \exp(-x)$,
or in other words
\begin{equation}
  \label{Gammaid}
  \Gamma(\beta; x) = \beta^{-1} (\Gamma(\beta + 1; x) -  x^{\beta} \exp(-x)).
\end{equation}
The one-dimensional characteristic function of $X_1$ uniquely
determines the term
\begin{align*}
  \int_{\bbR} \abs{g_\xi(x)}^{\beta} \Di x &= \int_0^{\infty} (1+\theta x)^{\beta} \exp(\lambda \beta x) \Di x
                                             = \left(\theta \exp(-\lambda \theta^{-1})\right)^{\beta} \int_{\theta^{-1}}^{\infty} y^{\beta} \exp(\lambda \beta y) \Di y
  \\
                                           &= -\frac{1}{\lambda \beta}\left(-\frac{\theta \exp(-\lambda \theta^{-1})}{\lambda \beta} \right)^{\beta} \Gamma(\beta+1; -\lambda \beta \theta^{-1})
                                             \eqqcolon c. 
\end{align*}
Now, we compute the covariation $[X_1, X_0]_{\beta}$:
\begin{align*}
  [X_1, X_0]_{\beta} &= \int_{\bbR} g_\xi(x + 1) g_\xi(x)^{\beta -1} \Di x
  \\
                     &= \int_0^{\infty} (1 + \theta (x + 1)) \exp(\lambda(x + 1)) (1 + \theta x)^{\beta - 1} \exp(\lambda (\beta - 1)x) \Di x
  \\
                     &= -\frac{1}{\lambda \beta} \left(-\frac{\theta \exp(-\lambda \theta^{-1})}{\lambda \beta} \right)^{\beta} \exp(\lambda) 
                       \left(  \Gamma(\beta + 1; -\lambda \beta \theta^{-1}) -\lambda \beta
                       \Gamma(\beta; -\lambda \beta \theta^{-1})   \right)
  \\
                     &= \exp(\lambda) (c(1 - \lambda) - \beta^{-1}),
\end{align*}
where we used the formula \eqref{Gammaid}. Since $c$ is uniquely
determined, the quantity $[X_1, X_0]_{\beta} $ identifies the
parameter $\lambda$ (note that $-c \lambda - \beta^{-1} > 0$, and in
particular this term is never equal to
0). Condition~(A)\cref{it:ass:A:4} is shown similarly to the previous
example.

\section*{Acknowledgement}
The authors acknowledge financial support from the project
\enquote{Ambit fields: probabilistic properties and statistical
  inference} funded by Villum Fonden.

\printbibliography

\end{document}

%%% Local Variables:
%%% mode: latex
%%% TeX-master: t
%%% End: